\documentclass[10pt,a4paper]{article}
\usepackage{amsmath, amsthm, amssymb,amsfonts,mathrsfs,amscd,latexsym}
\usepackage[usenames,dvipsnames]{xcolor}

\usepackage{makeidx}
\makeindex

\newtheorem{thm}{Theorem}[section]
\newtheorem{cor}[thm]{Corollary}
\newtheorem{lem}[thm]{Lemma}
\newtheorem{prop}[thm]{Proposition}

\theoremstyle{definition}
\newtheorem{defn}[thm]{Definition}
\newtheorem{rk}[thm]{Remark}
\newtheorem{ex}[thm]{Example}

\DeclareMathOperator{\Id}{Id}                 
\DeclareMathOperator{\Op}{Op}      
\DeclareMathOperator{\Res}{Res}         
\DeclareMathOperator{\res}{res}         
\DeclareMathOperator{\restor}{res^{\mathrm{tor}}}         
\DeclareMathOperator{\Real}{Re} 
\DeclareMathOperator{\supp}{supp}      
\DeclareMathOperator{\Tr}{Tr}                 
\DeclareMathOperator{\TR}{TR}                 
\DeclareMathOperator{\cutoff}{cut--off}

\newcommand{\A}{\mathcal{A}}              

\renewcommand{\a}{\alpha}                    
\renewcommand{\b}{\beta}                       
\newcommand{\B}{\mathcal{B}}              
\newcommand{\C}{\mathbb{C}}              
\newcommand{\del}{\partial}                    
\newcommand{\DD}{\mathcal{D}}           
\newcommand{\F}{\mathcal{F}}
\newcommand{\ga}{\gamma}                   
\newcommand{\N}{\mathbb{N}}
\newcommand{\Z}{\mathbb{Z}}
\newcommand{\Q}{\mathbb{Q}}
\renewcommand{\H}{\mathcal{H}}
\newcommand{\R}{\mathbb{R}}
\newcommand{\T}{\mathbb{T}}
\renewcommand {\S} {\mathbb S}

\newcommand{\Db}{\mathbf{\Delta}} 
\newcommand{\cutoffint}{-\hskip -10pt\int}
\newcommand{\cutoffsum}{-\hskip -4mm\sum}
\newcommand{\cutoffsumtor}{\textstyle{\cutoffsum_{\Z^n}^{\mathrm{tor}}}}
\renewcommand{\t}{\mathbf{t}}
\renewcommand{\k}{\mathbf{k}}

\newcommand {\e} {{\epsilon}}

\newcommand{\set}[1]{\{\,#1\,\}}               

\newcommand {\Ci} {{C^\infty}}

\newcommand {\Cl} {{C\ell}}
 
\newcommand{\wh}{\widehat}
\newcommand{\wt}{\widetilde}

\newcommand{\la}{\lambda}                

\newcommand{\norm}[1]{\left\lVert#1\right\rVert}    
\newcommand{\ol}{\overline}                  

\newcommand{\fp}{{\rm f.p.}}
 
\newcommand{\sg}{\sigma}
\renewcommand{\SS}{\mathcal{S}}

\newcommand{\Sbb}{\mathbb{S}}  

\renewcommand{\th}{\theta} 
\newcommand{\Tbb}{\mathbb{T}}

\newcommand{\twobytwosmall}[4]{\bigl(\begin{smallmatrix}#1 & #2 \\ #3 & #4
                            \end{smallmatrix}\bigr)} 

\begin{document}

\title{\bf The canonical trace and the noncommutative residue on the noncommutative torus}
\author{Cyril L\'evy, Carolina Neira Jim\'enez, Sylvie Paycha}

\maketitle

\begin{abstract}
Using a global symbol calculus for pseudodifferential operators on tori, we build a canonical trace on classical pseudodifferential operators on noncommutative tori in terms of a canonical discrete sum on the underlying toroidal symbols. We characterise the canonical trace on operators on the noncommutative torus as well as its underlying canonical discrete sum on symbols of fixed (resp.\ any) non--integer order. On the grounds of this uniqueness result, we prove that in the commutative setup, this canonical trace on the noncommutative torus reduces to Kontsevich and Vishik's canonical trace which is thereby identified with a discrete sum. A similar characterisation for the noncommutative residue on noncommutative tori as the unique trace which vanishes on trace--class operators generalises Fathizadeh and Wong's characterisation in so far as it includes the case of operators of fixed integer order. By means of the canonical trace, we derive defect formulae for regularized traces. The conformal invariance of the $\zeta$--function at zero of the Laplacian on the noncommutative torus is then a straightforward consequence. 
\end{abstract}

\tableofcontents

\section{Introduction}

Pseudodifferential operators on smooth manifolds are treated locally: to a local chart, one can associate a symbol of a given pseudodifferential operator as a smooth map on an open subset of $\R^n$. Only the local structure of $\R^n$ is used and there is no global notion of symbol of a pseudodifferential operator. This approach is natural for general smooth manifolds where one can hardly avoid local coordinates to extract geometrical information. However, on manifolds carrying more symmetries (Lie groups, homogeneous spaces) one can use this extra data to develop a richer, and global notion of symbol calculus of pseudodifferential operators \cite{RT4,NR}. 

In 1979, Agranovich \cite{A1} introduced such a calculus for pseudodifferential operators on the circle $\Sbb^1$, using Fourier series, and launched the notion of periodic symbol of pseudodifferential operators on the torus $\Tbb^n$. The general idea for the periodic quantisation on the torus can be summarized in the following way. If $a\in C^\infty(\Tbb^n)$,  one defines the discrete Fourier transform of $a$ as a function on the lattice $\Z^n$ (the Pontryagin dual of $\T^n$)
\[
 \F_{\Tbb^n}(a) (k) := \int_{\Tbb^n} e^{-2\pi i x\cdot k} a(x) dx.
\]

One then discretises the problem by using this Fourier transform instead of the Euclidean one in the very definition of a pseudodifferential operator on the torus. More precisely, the quantisation map is defined as
\[
 \Op(\sg) : a\mapsto \sum_{k\in \Z^n} e_k\, \sg(k) \, \F_{\T^n}(a)(k)
\] 
where $e_k(x):=e^{2\pi i x\cdot k}$. Operators of this type were called \emph{periodic} pseudodifferential operators. It turns out, and it is non--trivial, that periodic pseudodifferential operators actually coincide with pseudodifferential operators on the torus seen as a closed manifold \cite{A2,M,McL,TV}. What is actually new here, compared to the classical pseudodifferential calculus, is the possibility to invert the quantisation map, as it is injective, which leads to  a global (periodic) symbol calculus of pseudodifferential operators on the torus. Namely, if $A$ is a pseudodifferential operator, then the (global) symbol of $A$ is 
\[
\sg_A : \Z^n \to C^\infty(\T^n)\, , \quad k\mapsto A(e_k) e_{-k}\, .
\] 
Naturally, symbols on the torus are not maps from $\R^n$ to $C^\infty(\Tbb^n)$ (in contrast to the Euclidean case), but are actually defined on the Pontryagin dual $\Z^n$ of the torus. 
This discretization of the notion of symbol, which itself comes from the compactness of the torus as an abelian Lie group, calls for discrete--type analytic tools (finite difference operators, finite difference Leibniz formulae, etc.), which differ from the ones used in the global calculus on the Euclidean space $\R^n$ (see \cite{NR}).

The Lie group structure of the torus allows to apply harmonic analysis techniques directly to the pseudodifferential calculus. Such techniques used in the case of a torus, as well as extensions to other Lie groups ($SU(2)$ for example) have been further investigated by Ruzhansky and Turunen in \cite{RT1,RT2,RT3,RT4,T}. Aside from its elegance, the global calculus approach is useful for it has applications in hyperbolic partial differential equations, global hypoellipticity, $L^2$-boundedness, and numerical analysis (see \cite{RT4} and references therein). 

The goal of this article is to investigate traces on the global pseudodifferential calculus in the situation where the underlying manifold is now a noncommutative geometrical object, namely the noncommutative torus.

Connes' definition of noncommutative (compact, spin) manifold is based on the notion of spectral triple \cite{C2}. If $(\A,\H,D)$ is a spectral triple, $\A$ plays the role of the coordinate algebra (of smooth functions on the manifold), $\H$ is the Hilbert space of spinors, and $D$ is the (abstract) Dirac operator acting on $\H$. The idea is that the algebra $\A$ is not necessarily commutative. 
There are many examples of noncommutative spaces, and the noncommutative torus, described by the Fr\'echet algebra $\A_\th$, where $\th$ is the \emph{deformation matrix}, is probably the most simple one.
 
In 1980, Connes \cite{C1} defined a pseudodifferential calculus on the noncommutative torus, in the more general setting of $C^*$--dynamical systems (see also \cite{B1,B2}). The symbols of this calculus are maps from $\R^n$ into the algebra $\A_\th$. This calculus was used in \cite{CT} and \cite{FK1} to give a noncommutative version of the Gauss--Bonnet theorem, and more recently for the computation of the (noncommutative equivalent of the) scalar curvature \cite{CM2,FK2,FK3}. In \cite{FW} the notion of a noncommutative residue on classical pseudodifferential operators on the noncommutative two--torus was introduced, and it was proved that up to a constant multiple, it is the unique continuous trace on the algebra of such operators modulo infinitely smoothing operators.

However, since symbols are here defined on the whole space $\R^n$, the quantisation map is not injective as such. In order to recover a (global) symbol map, fully exploit techniques available on the space of symbols and construct non--singular traces, we modify the definition of symbols as in the commutative case. We define symbols as maps from the Pontryagin dual of $\T^n$, namely the standard lattice $\Z^n$, into the algebra $\A_\th$ of the noncommutative torus. Following the terminology of Ruzhansky and Turunen we call symbols on $\Z^n$ \emph{toroidal symbols} and their corresponding operators \emph{toroidal operators}.

We use this  global symbol calculus to construct and characterise the (noncommutative equivalent) of the canonical trace. Recall that on a closed smooth manifold the canonical trace is (up to a multiplicative factor) the unique linear extension of the ordinary trace to non trace--class classical pseudodifferential operators of non--integer order \cite{MSS} which vanishes on brackets in that class. On classical pseudodifferential operators of fixed non--integer order, a trace is a linear combination of the canonical trace and a singular trace called the leading symbol trace \cite{LN-J}. On the one hand, non--integer order operators build a class of operators on which the noncommutative residue vanishes but on the other hand, the canonical trace does not extend as a linear form to the algebra of integer order operators where the residue becomes a relevant linear form. This dichotomy  between the residue and the canonical trace was clarified in one of the authors' thesis work \cite{N-J} followed by \cite{LN-J} and carries out to the noncommutative setup as it shall become clear from our main classification result Theorem \ref{mainthm1}. In contrast with the canonical trace of Kontsevich and Vishik which is built from an integral of the symbol of  a pseudodifferential operator on a closed manifold, our canonical trace on   noncommutative tori is built from a discrete sum involving the symbol of a toroidal pseudodifferential operator. In the commutative setup, this global symbolic approach nevertheless leads to Kontsevich and Vishik's  canonical trace on ordinary tori seen as closed manifolds. Our results actually offer a generalisation of uniqueness results both from the commutative setup to the noncommutative setup and from the continuous to the toroidal setup. These characterisations are nevertheless derived under the assumption that the linear forms be either exotic (Definition \ref{exotic}) or $\ell^1$--continuous (Definition \ref{ell 1-continuous}), two assumptions that can probably be circumvented although they seem to be needed in our approach. In their classification of traces on the noncommutative two--torus which easily generalises to higher dimensional tori, Fathizadeh and Wong \cite{FW} required that the trace be singular and continuous instead of exotic. 

This paper is organised as follows. After some preliminaries on the noncommutative torus in Section \ref{secprem}, we extend in Section \ref{sectoroidal}  the global pseudodifferential calculus on the ordinary torus \cite{RT1} to noncommutative tori. This calculus has the remarkable feature that unlike the usual pseudodifferential calculus on closed manifolds, the quantisation map sets up a one to one correspondence between symbols and operators (Proposition \ref{propquantisation}). Via this bijection, the composition product on noncommutative toroidal pseudodifferential operators yields a star--product (\ref{starproduct}) on noncommutative toroidal symbols, just as the Weyl--Moyal product can be derived from the global Weyl quantisation map on $\R^n$.

We show (Theorem \ref{thmcomp}) that noncommutative toroidal pseudodifferential operators equipped with the composition of operators, build an algebra. Known regularity properties of pseudodifferential operators on compact manifolds generalise to toroidal pseudodifferential operators (Theorem \ref{thm:Sobolev}). In Section \ref{secextensionmap} we use smooth extensions of discrete symbols, a notion already introduced in the commutative setup \cite{RT1,RT4}. 
This allows to define the notion of extension map (Definition \ref{def:extension}) which associates to any discrete symbol $\sg$ a smooth extension of $\sg$. Extension maps provide a way us to transfer known concepts for symbols on $\R^n$ to noncommutative toroidal symbols such as quasihomogeneity (Proposition \ref{prop:quasihom}) and polyhomogeneity. In particular, we consider the subspace of noncommutative toroidal classical symbols (Definition \ref{defclass}) and  subclasses of that algebra such as the class of non--integer order classical noncommutative toroidal symbols and that of fixed order. We furthermore relate the star--product on toroidal symbols to the star--product on their extensions (Theorem \ref{thm:asympt}), a relation which is useful to prove traciality in the toroidal setup. 

Using an extension map, in Section \ref{sectraces} we build a canonical discrete sum on noncommutative toroidal non--integer order classical symbols from its integral counterpart on non--integer order classical symbols on $\R^n$. Throughout the paper we assume that $n>1$. Since traciality of linear maps on toroidal symbols implies $\Z^n$--translation invariance (Lemma \ref{lem:translationinv}), we derive the characterisation of traces from that of $\Z^n$--translation invariant linear forms on toroidal symbols. The corresponding uniqueness results (Theorems \ref{thm:charres} and \ref{thm:characressum}) are new to our knowledge and  interesting in their own right. The characterisation of $\Z^n$--translation invariant linear forms on toroidal symbols is in turn derived via an extension map from the characterisation of $\Z^n$--translation invariant linear forms on symbols on $\R^n$ (Proposition \ref{prop:charresRn}) already investigated in \cite[Proposition 5.40]{P2}.

The fact that the coefficient algebra is the noncommutative torus plays a decisive role in our approach. Indeed, the specific form of the derivations $\delta_j$ on ${\mathcal \A}_\theta$ given by (\ref{smalldeltaj}) is strongly used all along the paper, as e.g.\ in Lemma \ref{lem1} to prove the absolute summability of the quantization map and in Theorem \ref{thmcomp} to derive the star product $\circ_\theta$ from the composition of two quantized symbols. Also, the ``commutativity" relation for the Weyl elements given by formula (\ref{eq:UkUl}) is essential in the proof of the key Lemma \ref{lem:translationinv} which relates traciality to closedness and translation invariance of linear forms on toroidal symbols.

Section \ref{section: classification nct} presents our main result (Theorem \ref{mainthm1}), namely the characterisation of the canonical discrete sum (resp.\ the noncommutative residue) on noncommutative toroidal non--integer (resp.\ integer) order classical symbols and of the corresponding canonical trace (resp.\ noncommutative residue) on noncommutative toroidal non--integer (resp.\ integer) order classical pseudodifferential operators. Along with these results we provide refined characterisations on symbols and operators of fixed order similar to the ones derived in \cite{N-J} and \cite{LN-J}. The commutative counterpart of Theorem \ref{mainthm1} stated in Corollary \ref{maincor1} yields a characterisation of traces on toroidal symbols of fixed order. It also yields back known characterisations of the noncommutative residue \cite{W1,W2} and the canonical trace \cite{KV,LN-J,MSS,N-J} on certain classes of pseudodifferential operators on the torus seen as a particular closed manifold. In particular, this uniqueness result provides an alternative description of the canonical trace on tori in terms of a canonical discrete sum already investigated from another point of view in \cite{P2}.

The strategy that we follow for the proof of Theorem \ref{mainthm1} is based on several steps. First, we observe that the classification on the operator level is a direct consequence of the one on the symbol level since the quantisation map is an algebraic and topological isomorphism between noncommutative toroidal symbols and operators (Proposition \ref{propquantisation} and Theorem \ref{thmcomp}). In Section \ref{subsectransinv} we show that traces on noncommutative toroidal symbol spaces are closed and $\Z^n$--translation invariant (Lemma \ref{lem:translationinv}). This way we can reduce  the problem to a (commutative) classification of $\Z^n$--translation invariant linear forms on subsets of ordinary toroidal symbols. This classification described in Section \ref{subsecclassif} is interesting for its own sake and relies on an extension procedure from toroidal symbols to symbols on $\R^n$ combined with a classification of $\Z^n$--translation invariant linear forms on  symbols  on $\R^n$ (Theorem \ref{thm:charres} and Theorem \ref{thm:characressum}). This yields the (projective) uniqueness part of the theorem: any exotic trace on the algebra of integer order noncommutative toroidal classical symbols is proportional to the noncommutative residue whereas a trace on non--integer order noncommutative toroidal classical symbols which is continuous on $L^1$--symbols is proportional to the canonical sum. In the final step we check the tracial properties for these linear forms (Propositions \ref{prop:noncommutative residue trace}  and \ref{prop:canonical trace}).

The canonical trace on the noncommutative torus that we have characterised this way, provides a building block to construct a holomorphic calculus on pseudodifferential operators on the noncommutative torus. This is carried out in Section \ref{section:holomorphic} which presents an application of the previous constructions in so far as it provides (Theorem \ref{thm:KVPS}) an interpretation of the residue as a complex residue (formula (\ref{KV})) and gives an explicit formula of the finite part at zero of traces of holomorpic families in terms of a residue (formula (\ref{PSOp})). It further yields an expression (Corollary \ref{cor:defect}) for the obstruction that prevents linear extensions (see (\ref{eq:zetatrace})) of the ordinary trace built from a zeta regularisation procedure from defining traces and for the dependence on the regulator one uses to define these extensions. Although similar to the known ones in the usual pseudodifferential calculus on closed manifolds (see \cite{PS}), these formulae are not straightforward in our discretised set up since they require an appropriate notion of holomorphicity for families of toroidal symbols and operators (Definition \ref{def:holomorphicfamily}) compatible with the usual notion of holomorphic family of ordinary symbols (Proposition \ref{prop:holdiscreteversuscont}). From these defect formulae, we easily recover the conformal invariance of the $\zeta$--function at zero of the Laplacian shown in \cite{CT}, \cite{CM1,CM2} and \cite{FK1,FK2}. \\
Having this holomorphic calculus at hand raises the question how the canonical trace built in the present paper on pseudodifferential operators on the noncommutative torus $A_\theta$ relates to the canonical trace associated with the corresponding spectral triple
\[
 (A_\theta,\H\otimes \C^{2^{[n/2]}},\DD:=\delta_j \otimes \ga^j)
\] 
(see \eqref{spectraltriple} below) as built in \cite{P3}, an interesting issue for possible future work.

\section{Preliminaries on the noncommutative torus}
\label{secprem}

Let $\th$ be an $n\times n$ antisymmetric real matrix. Let $A_\th$\index{$A_\th$} denote the twisted group $C^*$--algebra $C^*(\Z^n,c)$ where $c$ is the following 2--cocycle for the abelian group $\Z^n$:
 \[
  c(k,l) = e^{- \pi i\,\langle k,\th l\rangle}\,,\qquad k,l\in \Z^n\, .
 \]

Recall that $C^*(\Z^n,c)$ is the enveloping $C^*$-algebra of the Banach twisted--convolution algebra $L^1(\Z^n,c)$\index{$L^1(\Z^n,c)$}. 

\begin{defn}
\label{def:coefficients}
 The \emph{coefficients} of an element $a$ of the algebra $L^1(\Z^n,c)$ are the elements of the unique sequence $\{a_k\}_{k\in \Z^n}$ in $\C$ such that $a$ can be decomposed as the convergent series $a=\sum_{k\in \Z^n} a_k U_k$ where the $(U_k)$\index{$U_k$} are the \emph{Weyl elements}. The Weyl elements are unitaries in $A_\th$ that satisfy $U_0=1$ and
\begin{equation}
\label{eq:UkUl}
 U_k U_{l} = c(k,l)U_{k+l}.
\end{equation}
\end{defn}

Note that $c(k,l)\, c(l,k)=1$ so that  
\[
 U_k U_{l} = \frac{c(k,l)}{ c(l,k) }U_lU_k= e^{-2 \pi i\,\langle k,\th l\rangle}U_lU_k.
\]
Thus when $\theta$ has integer entries, the Weyl elements commute.\\
Let us further observe that $U_k=e^{i\pi \sum_{l<m}k_l \theta_{lm} k_m}e_k$, where $e_k(x):=e^{2\pi i \langle x,k\rangle}$ is the $k$--th phase function \cite[Section 12.2]{G-BVF}.

We define for all $a\in L^1(\Z^n,c)$,
\begin{equation}\label{trace}
 \t(a) := a_0 \index{$\t$}
\end{equation}
and extend (by norm continuity) $\t$ as a (normalised) trace on $A_\th$.\\

\begin{rk}
 When $n=2$ and $\theta=\twobytwosmall{0}{\th_0}{-\th_0}{0}$ where $\th_0\notin \Q$, $\t$ is the unique normalised trace on $A_\th$ \cite[Corollary 50]{C3}. For general $n$ the trace $\t$ on $A_\th$ is unique whenever $\th$ is quite irrational with the definition of \cite[p. 537]{G-BVF}. However, the results in the following sections do not require any condition on $\th$ since 
 they do not rely on any uniqueness result for $\t$. 
\end{rk}

Let $\A_\th$\index{$\A_\th$} denote the involutive subalgebra of $L^1(\Z^n,c)$ consisting of series of the form $\sum_k a_k U_k$ where $(a_k)\in \SS(\Z^n)$, the vector space of sequences $(a_k)$ that decay faster than the inverse of any polynomial in $k$. We fix the following inner product on $A_\th$:
\[
 \langle a,b\rangle:=\t(ab^*) \qquad \forall\, a,b\in A_\th\, .
\]
Let $\H$ be the GNS Hilbert space corresponding to the previous inner product. The associated GNS representation $\pi$ yields an $n$--dimensional regular spectral triple which is the \emph{noncommutative $n$--torus with deformation matrix} $\th$:
\begin{equation}
\label{spectraltriple}
 (A_\th,\H\otimes \C^{2^{[n/2]}},\DD:=\delta_j \otimes \ga^j)
\end{equation}
where $A_\th$ acts as $\pi(a)\otimes \Id$ on $\H\otimes\C^{2^{[n/2]}} $ and where for all $j \in \set{1,\cdots,n}$, $\delta_j$ is the derivation on $\A_\th$ given by
 
\begin{equation}
 \label{smalldeltaj}
 \delta_j \left(\sum_{k\in \Z^n} a_k U_k\right):= \sum_{k\in \Z^n} a_k   k_j  U_k\, ,\index{$\delta_j$}
\end{equation}
considered as a densely defined operator in $\H$. The $\gamma^j$, $j \in \set{1,\cdots,n}$, stand for the Dirac matrices. The fact that $\t \circ \delta_j\ = 0$ (on $\A_\th$) for all $j\in \set{1,\cdots,n}$ will play an important role in the following.

The algebra $\A_\th$ can also be seen as the smooth elements of $A_\th$, for the continuous action of the torus $\Tbb^n=\R^n/\Z^n$ on $A_\th$ defined on the unitaries $(U_k)_k$ by $\a_s(U_k):=e^{2\pi i \,\langle s, k\rangle} U_k$, where $s\in \R^n$. The infinitesimal generators of this action are precisely the derivations $(2\pi i\delta_j)_j$.
Using these derivations, we equip $\A_\th$ with a structure of Fr\'echet $*$--algebra where the topology is given by the following seminorms:
\[
 p_\a(a) := \norm{ \delta^\a (a)}\, , \quad \a\in \N^n, \index{$p_\a$}
\]
where $\delta^\a := \delta_1^{\a_1}\cdots \delta_n^{\a_n}$\index{$\delta^\a$}, and $\norm{\cdot}$ is the norm associated to the scalar product $\langle \cdot,\cdot\rangle$. Let $\mathbf{\Delta}$ denote the operator $\sum_j \delta_j^2$ on $\A_\th$. We will also use the notation $\langle \xi\rangle:= \sqrt{|\xi|^2+1}$ for all $\xi\in \R^n$.

\begin{rk}
\label{rk:A=A0}
When $\th$ has integer entries, $\A_\th=\A_0$ is isomorphic to the (commutative) algebra (under pointwise multiplication) of smooth functions on the (commutative) torus $\A:=C^\infty(\Tbb^n)$. 
\end{rk}

\begin{lem}
\label{lemsemi}
 The seminorms $q_N$ ($N\in \N$), given by $q_{N}(a):= \sup_{k\in \Z^n} \langle k\rangle^N |a_k|$\index{$q_{N}$} for all $a=\sum_{k} a_k U_k\in \A_\th$, corresponding to $\SS(\Z^n)$, yield the same topology as the seminorms $p_{\a}$.
\end{lem}

\begin{proof}
  Let $\a\in \N^n$ and $a\in \A_\th$. We have $p_\a(a)\leq \sum_{k\in \Z^n} |a_k| \langle k \rangle^{|\a|}$, which implies the estimate
  \[
  p_\a(a)\leq \left(\sum_{k\in \Z^n}\langle k\rangle^{-n-1}\right) \, q_{|\a|+n+1}(a)\, .
  \]
  Let $N\in \N$. Since $\langle (1+\Db)^N (a),U_k\rangle= \langle k\rangle^{2N} a_k$, we get
 \[
  q_N(a)\leq q_{2N}(a) \leq \norm{(1+\Db)^N(a)}\,,
 \]
  so the result follows.
\end{proof}

\begin{defn} 
Let $j\in \set{1,\cdots,n}$ and $B$ be a given algebra. The \emph{forward difference operator} $\Delta_{j}$ is the linear map $B^{\Z^n}\to B^{\Z^n}$ defined by
\begin{equation}\label{bigdeltaj}
 \Delta_j (\sg) (k) := \sg(k+e_j)-\sg(k) \index{$\Delta_{j}$}
\end{equation}
where $(e_j)_{1\leq j \leq n}$ is the canonical basis of $\R^n$.

If $\a\in \N^n$, we set $\Delta^\a:=\Delta_1^{\a_1}\cdots \Delta_n^{\a_n}$,\index{$\Delta^\a$} which is also denoted by $\Delta_k^\a$ to specify the relevant variable. 
\end{defn}

\begin{rk}
It is a feature of the calculus of finite differences that $\Delta_j$ is \emph{not} a derivation of the algebra (with pointwise product) $B^{\Z^n}$. Indeed, if $\sg,\tau\in B^{\Z^n}$, then $\Delta_j (\sg\tau) =\Delta_j(\sg) T_{e_j}(\tau) +\sg \Delta_j(\tau)$, where $T_l(\tau):=\tau(\cdot+l)$.\index{$T_l$} However, there is a Leibniz formula adapted to this calculus (see \cite[Lemma 3.3.6]{RT4}, the proof extends directly to functions valued in arbitrary algebras): if $\sg,\tau\in B^{\Z^n}$, then
 \begin{equation}
  \label{LeibnizDisc}
  \Delta^\a (\sg\tau) = \sum_{\b\leq \a}\tbinom{\a}{\b} \Delta^\b(\sg) T_{\b} \Delta^{\a-\b} (\tau)\, .
  \end{equation}
\end{rk}

For a function $\sg :\Z^n\to\A_\th$, and $l\in \Z^n$, let $\sg_l$ denote the map from $\Z^n$ into $\C$ given by $\sg_l(k):= (\sg(k))_l$ (see Definition \ref{def:coefficients}). Hence, for any $k\in \Z^n$, $\sg(k) = \sum_l \sg_l(k) U_l$.

\section{Toroidal symbols and associated operators}
\label{sectoroidal}
 
\subsection{Toroidal symbols}

\begin{defn} Let $\B$ be a Fr\'echet algebra whose topology is associated with a given countable family of seminorms $(p_i)_{i\in I}$. A function $\sg : \Z^n \to \B$ is a \emph{(discrete) toroidal symbol of order} $m\in \R$ on $\B$, if for a countable set $I$, for all $(i,\b) \in I\times \N^n$, there is a constant $C_{i,\b}\in \R$, such that for all $k \in \Z^n$,
 \begin{equation}\label{pdeltasigma}
 p_i(\Delta^\b \sg (k)) \leq C_{i,\b} \langle k\rangle^{m-|\b|}\,.
 \end{equation}
The space of all discrete symbols of order $m$ on $\B$ is denoted by $S^m_{\B}(\Z^n)$.\index{$S^m_{\B}(\Z^n)$}
We define similarly the space of \emph{(smooth) toroidal symbols} $S^m_\B(\R^n)$\index{$S^m_{\B}(\R^n)$} on $\B$ by supposing $\sg\in C^\infty(\R^n,\B)$, and replacing $\Delta^\b$ by the usual operator $\del_\xi^\b$:
\begin{equation*}
 p_i(\del_\xi^\b \sg (\xi)) \leq C_{i,\b} \langle \xi\rangle^{m-|\b|}\, , \  \forall \xi\in\R^n.
\end{equation*}
We further define the space of all toroidal symbols on $\B$ as $S_{\B}(\Z^n):=\cup_{m\in \R} S^m_{\B}(\Z^n)$\index{$S_{\B}(\Z^n)$}, which is, for $\B=\A_\th$ by the discrete Leibniz formula \eqref{LeibnizDisc}, an $\R$--graded algebra under pointwise multiplication. For $\B=\C$, the space  $S_{\B}(\Z^n)$ corresponds to the space of symbols with constant coefficients. The ideal of smoothing symbols is $S^{-\infty}_{\B}(\Z^n):=\cap_{m\in \R} S^m_{\B}(\Z^n)$.\index{$S^{-\infty}_{\B}(\Z^n)$} The spaces $S_{\B}(\R^n)$\index{$S_{\B}(\R^n)$} and $S^{-\infty}_{\B}(\R^n)$\index{$S^{-\infty}_{\B}(\R^n)$} are defined similarly.
\end{defn}

In the following we shall mainly be concerned with the following symbol spaces: $S^m_{\A_\th}(\Z^n)$, $S^m_{\A_\th}(\R^n)$, $S^m_\C(\Z^n)$ and $S^m_\C(\R^n)$.

\begin{rk} 
\label{rk:symbol spaces}
For $\A=\A_0$ we have $C^\infty(\R^n,C^\infty(\T^n))\simeq C^\infty(\R^n\times \T^n)$, so that $S^m_\A(\R^n)$ is the usual symbol space on the commutative torus (see Remark \ref{rk:A=A0}). Similarly, $S^m_\C(\R^n)$ is the usual space of symbols that are independent of the variable $x$ on the commutative torus. 
\end{rk}

\begin{ex} 
If $j\in \set{1,\cdots,n}$, the map $k\mapsto k_j U_0$ is a symbol in $S^1_{\A_\th}(\Z^n)$. 
Moreover, any element of $\A_\th$ can be seen as a symbol in $S^0_{\A_\th}(\Z^n)$, through the injection $\A_\th\to S^0_{\A_\th}(\Z^n)$ given by $a\mapsto (k\mapsto a)$.
\end{ex}

The space $S^m_{\A_\th}(\Z^n)$ is a Fr\'echet space for the seminorms
\begin{equation}\label{Frechetseminorms}
 p_{\a,\b}^{(m)}(\sg):=\sup_{k\in \Z^n} {\langle k\rangle^{-m+|\b|}}p_\a(\Delta^\b \sg(k))\, . \index{$p_{\a,\b}^{(m)}$}
\end{equation}

We have the following relation between discrete and smooth symbols:
\begin{lem}
\label{restr} 
Let $\B$ be either $\A_\th$ or $\C$. The restriction map $r:\B^{\R^n}\to \B^{\Z^n}$, $\sg\mapsto \sg_{|\Z^n}$ maps $S_\B^m(\R^n)$ into $S_\B^m(\Z^n)$ for all $m\in \R$. In particular, it sends smoothing symbols to smoothing discrete symbols.
\end{lem}

\begin{proof} 
The proof is similar to the proof of the ``if" part of \cite[Theorem 4.5.2]{RT4}.
\end{proof}

There is also a relation between symbols with values in $\A_\th$ and complex valued symbols:

\begin{lem}
\label{lem:fouriercoeff}
Let $\sg\in S^m_{\A_\th}(\R^n)$. Then for any $l\in \Z^n$, the function $\sg_l$ defined by $\xi\mapsto \t( \sg(\xi)U_{-l})$ belongs to $S^m_{\C}(\R^n)$. Moreover, for any $\b \in \N^n$ and $N\in \N$, there is a constant $C_{\b,N}>0$ such that for all $\xi\in \R^n$
\[
|\del_\xi^\b\sg_l(\xi)| \leq C_{\b,N} \langle \xi\rangle^{m-|\b|} \langle l\rangle^{-N} \, .
\]
The same properties hold for discrete symbols, replacing $\del_\xi^\b$ by difference operators.
\end{lem}

\begin{proof}
Let $N\in \N$, $\b\in \N^n$, and $\sg\in S^m_{\A_\th}(\R^n)$. From $\langle \delta_j(a),b\rangle=\langle a,\delta_j(b)\rangle$ for all $a,b\in \A_\th$, and $j\in \set{1,\cdots,n}$, we deduce that for all $a\in \A_\th$, and $l\in \Z^n$,
$\langle a,U_l\rangle= \langle l\rangle^{-2N}\sum_{|\mu|\leq 2N} c_{\mu,N} \langle \delta^\mu (a),U_l\rangle$ where the $c_{\mu,N}$ are positive coefficients such that $\langle l\rangle^{2N}=\sum_{|\mu|\leq 2N} c_{\mu,N} l^\mu$ for all 
$l\in \Z^n$. This yields the following estimate for all $\xi\in \R^n$:
\begin{align*}
|\del_\xi^\b \sg_l(\xi)| = |\langle \del_\xi^\b \sg(\xi),U_l\rangle|&\leq \langle l\rangle^{-2N}\sum_{|\mu|\leq 2N} c_{\mu,N} \norm{\delta^\mu( \del_\xi^\b \sg(\xi))}\\
 &\leq C_{\b,N}\langle  \xi\rangle^{m-|\b|} \langle l\rangle^{-2N}
\end{align*}
where $C_{\b,N}:=\sum_{|\mu|\leq 2N} c_{\mu,N} p_{\mu,\b}^{(m)}(\sg)$. The case of discrete symbols is similar.
\end{proof}

Let $\sg\in S_{\A_\th}(\Z^n)$ and $\sg_{[j]} \in S_{\A_\th}^{m_j}(\Z^n)$ for $j\in \N$ where $m_j\in \R$, $m_j> m_{j+1}$, and $\lim_{j\to\infty} m_j =-\infty$.
As in the commutative toroidal calculus, the notation $\sg \sim \sum_{j=0}^\infty \sg_{[j]}$ means that $\sg-\sum_{j=0}^N \sg_{[j]}\in S^{m_{N+1}}_{\A_\th}(\Z^n)$ for all $N\in \N$.

If $\sg, \tau\in S^m_{\A_\th}(\Z^n)$, the notation $\sg\sim \tau$ means that $\sg-\tau\in S^{-\infty}_{\A_\th}(\Z^n)$.

We extend the previous notations to the case of smooth symbols on the noncommutative torus, i.e.\ when $\sg$, $\tau$, $\sg_{[j]}$ belong
to $S_{\A_\th}(\R^n)$. 

As in the commutative toroidal calculus, it is possible to build symbols from asymptotics:
\begin{lem} 
\label{lemasym}
 Let $\B$ be either $\A_\th$ or $\C$. If $\sg_{[j]} \in S_{\B}^{m_j}(\Z^n)$ (resp.\ $S_{\B}^{m_j}(\R^n)$) for $j\in \N$ where $m_j\in \R$, $m_j> m_{j+1}$, and $\lim_{j\to\infty} m_j =-\infty$,
 then there exists $\sg\in S^{m_0}_{\B}(\Z^n)$ (resp.\ $S^{m_0}_{\B}(\R^n)$) such that $\sg\sim \sum_{j=0}^\infty \sg_{[j]}$.
\end{lem}

\begin{proof} 
For the case of smooth symbols, the proof is similar to the standard (commutative) case of symbols on $\R^n$, and for the case of discrete symbols, the proof is similar to \cite[Theorem 4.1.1]{RT4}.
\end{proof}

\begin{defn}
\label{defn:bar t}
We define $\bar \t$ as a continuous linear map $\bar \t: S^m_{\A_\th}(\Z^n)\to S^m_\C(\Z^n)$\index{$\bar \t$} given by 
\[
 \bar \t (\sg): k\mapsto \t(\sg(k)) = \sg_0(k) \, ,
\]
where $\t$ is the trace defined in \eqref{trace}.
This map is compatible with the natural injection $\iota_\th: S^m_\C(\Z^n)\to S^m_{\A_\th}(\Z^n)$\index{$\iota_\th$} defined by 
\[
\iota_\th (\tau) : k\mapsto \tau(k) U_0
\]
in the sense that $\bar \t \circ \iota_\th = \Id_{S^m_\C(\Z^n)}$.

We define a similar map on smooth symbols (i.e.\ from the space $S^m_{\A_\th}(\R^n)$ to the space $S^m_\C(\R^n)$), still denoted by $\ol \t$, and the natural injection from complex valued smooth symbols $S^m_\C(\R^n)$ to $\A_\th$--valued smooth symbols $S^m_{\A_\th}(\R^n)$ is still denoted by $\iota_\th$. 
\end{defn}

\subsection{Toroidal pseudodifferential operators}

In order to introduce the notion of toroidal pseudodifferential operators we need to define a quantisation map.
\begin{lem}
\label{lem1}
(i) Let $a\in \A_\th$ and $\sg\in S^m_{\A_\th}(\Z^n)$. Then
  \[
  \Op_\th(\sg)(a):=\sum_{k\in \Z^n} a_k \, \sg(k)\,  U_k \index{$\Op_\th$}
  \]
  is absolutely summable in $\A_\th$.

\medskip

(ii) If $\sg\in S^m_{\A_\th}(\Z^n)$, the linear operator $\Op_\th(\sg):a\mapsto \Op_\th(\sg)(a)$ is continuous from $\A_\th$ into itself.

\medskip

(iii) If $a\in \A_\th$, the linear operator $\Op_\th (\cdot) (a) :\sg \mapsto \Op_\th(\sg)(a)$ is continuous from $S^m_{\A_\th}(\Z^n)$ into $\A_\th$.
\end{lem}

\begin{proof}
  Let $\a\in \N^n$. The Leibniz formula 
  \[
  \delta^\alpha\left(a_k\, \sg(k)\, U_k\right)=\sum_{\gamma+\gamma^\prime=\alpha} \tbinom{\alpha}{\gamma}\delta^\gamma(\sigma(k))\, \delta^{\gamma^\prime} (a_k U_k)
  \] 
  combined with  \eqref{pdeltasigma} yields the existence of a constant $C_\a>0$ such that for all $k\in \Z^n$,
  \[
  p_\a(a_k\, \sg(k)\, U_k) \leq C_\a |a_k| \langle k\rangle^{|\a|+m} \sum_{\gamma\leq \a} p^{(m)}_{\gamma,0}(\sg)\, .
  \]
  As a consequence, we obtain for all $k\in \Z^n$:
  \begin{equation}
  \label{lem1est}
    p_\a(a_k\, \sg(k)\, U_k) \leq C_\a \langle k\rangle^{-n-1} q_{N}(a) \sum_{\ga\leq \a} p^{(m)}_{\ga,0}(\sg)\, 
  \end{equation}
  where $N\geq|\a|+m+n+1$. This yields $(i)$ and $(iii)$, and $(ii)$ follows from \eqref{lem1est} and Lemma \ref{lemsemi}.
\end{proof} 

\begin{defn} 
A \emph{toroidal pseudodifferential operator of order} $m$ is a continuous linear operator $\A_\th\to \A_\th$ of the form $\Op_\th(\sg)$ for a symbol $\sg\in S^m_{\A_\th}(\Z^n)$. We denote by $\Psi_\th^m(\T^n):=\Op_\th(S^m_{\A_\th}(\Z^n))$\index{$\Psi_\th^m(\T^n)$} the space of pseudodifferential operators of order $m$, and we further set $\Psi_\th(\T^n):=\cup_m \Psi^m_\th(\T^n)$\index{$\Psi_\th(\T^n)$}, $\Psi^{-\infty}_\th(\T^n):= \cap_m \Psi^m_\th(\T^n)$.\index{$\Psi^{-\infty}_\th(\T^n)$}
\end{defn}

\begin{rk}
The space $\Psi_0^m(\T^n):=\Op_0(S^m_{\A}(\Z^n))$ is the standard space $\Psi^m(\T^n)$ of pseudodifferential operators on the commutative torus \cite[Theorem 5.4.1]{RT4} (see Remark \ref{rk:symbol spaces}).
\end{rk}

One of the features of the toroidal calculus (as well as other global calculi) is the one to one correspondence between pseudodifferential operators and symbols.
This feature also holds in the noncommutative setting:

\begin{prop}
\label{propquantisation}
  (i) The quantisation map $\Op_\th: S^m_{\A_\th}(\Z^n) \to \Psi^m_\th(\T^n)$ is a bijection.

\medskip

  (ii) The inverse (dequantisation) map $\Op_\th^{-1}$ satisfies for all $A\in \Psi^m_\th(\T^n)$ and $k\in \Z^n$,
 \begin{equation}\label{definition:Op-1}
  \Op_\th^{-1} (A) (k)= A(U_k)\, U_{-k}\, .
 \end{equation}
\end{prop}

\begin{proof}
$(i)$ The linear map $\Op_\th$ is surjective by definition. If $\Op_\th(\sg)=0$, then in particular $\Op_\th(\sg)(U_k)=\sg(k)U_k=0$ for all $k\in \Z^n$. This implies that $\sg(k)=0$ for all $k\in \Z^n$, and so $\sg=0$.

$(ii)$ One easily checks that $\Op_\th \Op_\th^{-1} = \Id_{\Psi^m_\th(\T^n)}$ and $\Op_\th^{-1}\Op_\th = \Id_{S^m_{\A_\th}(\Z^n)}$, where $\Op_\th^{-1}$ is the linear map defined in $\eqref{definition:Op-1}$.
\end{proof}
 
The quantisation map $\Op_\th$ therefore extends to a bijective linear map  $\Op_\th: S_{\A_\th}(\Z^n)\to\Psi_\th(\T^n)$ compatible with the filtration and induces a bijective linear map $\Op_\th: S_{\A_\th}^{-\infty}(\Z^n)\to \Psi^{-\infty}_\th(\T^n)$. 
                               
As the following result shows, pseudodifferential operators can be composed and their composition is also a pseudodifferential operator. Transporting the composition product on $\Psi_\th(\T^n)$ over to the symbol space $S_{\A_\th}(\Z^n)$ yields a star--product on $S_{\A_\th}(\Z^n)$, just as the Weyl--Moyal product can be derived from the global Weyl quantisation map on $\R^n$.

\begin{thm}
\label{thmcomp}
  Let $A\in \Psi^{m}_\th(\T^n)$, and $B\in \Psi^{m'}_\th(\T^n)$. Then $AB\in \Psi^{m+m'}_\th(\T^n)$. More precisely, if $\sigma\in S_{\A_\th}^m(\Z^n)$ and $\tau\in S_{\A_\th}^{m^\prime}(\Z^n)$ then  
  \[
  \Op_\th(\sg)\, \Op_\th(\tau)= \Op_\th(\sg\circ_\theta\tau)
  \]
  where we have set
\begin{equation}\label{starproduct}
(\sg\circ_\theta \tau)(k) :=  \sum_{l\in \Z^n} \tau_{l}(k)\,  \sg(l+k)\,  U_l \, . \index{$\circ_\theta$}
\end{equation} 
The bilinear map $\circ_\th : S_{\A_\th}^m(\Z^n)\times S_{\A_\th}^{m'}(\Z^n) \to S_{\A_\th}^{m+m'}(\Z^n)$ is called the \emph{star--product} of $\sigma$ and $\tau$.\\
Consequently, $\Psi_\th(\T^n)$ is an $\R$--graded algebra under composition of operators. 
Moreover, $\Psi_\th^{-\infty}(\T^n)$ is an ideal of $\Psi_\th(\T^n)$. We call $\Psi_\th^{-\infty}(\T^n)$ the ideal of \emph{smoothing} operators.
\end{thm}

\begin{proof} 
We want to show that $\rho: k\mapsto AB(U_k)U_{-k}$ lies in $ S^{m+{m'}}_{\A_\th}(\Z^n)$. A straightforward computation shows that for any $k\in \Z^n$,
\[
\rho(k)= \sum_{l\in \Z^n} \tau_l(k)\, \sg(l+k)\, U_l=(\sg\circ_\theta \tau)(k) .
\]
Thus, it is enough to check that $\sum_l \rho^{(l)}$, where $\rho^{(l)}:k\mapsto \tau_l(k)\, \sg(l+k)\, U_l$, is absolutely summable in the Fr\'echet space $S^{m+{m'}}_{\A_\th}(\Z^n)$. Let $\a,\b\in \N^n$.
A computation based on the discrete Leibniz formula \eqref{LeibnizDisc} shows that for all $l,k\in \Z^n$,
\[
\delta^\a \Delta^\b_k \rho^{(l)} (k) =\sum_{\a'\leq \a}\sum_{\b'\leq \b} \tbinom{\a}{\a'}\tbinom{\b}{\b'}l^{\a-\a'}\langle \Delta_k^{\b'} \tau(k),U_l\rangle
\, \delta^{\a'}\Delta_k^{\b-\b'}\sg(l+k+\b')\, U_l\, .
\]
Let $N\in \N$, and write $\langle l\rangle^{2N}=\sum_{|\mu|\leq 2N} c_{\mu,N} l^\mu$ where $c_{\mu,N}$ are non--negative coefficients. Using the fact that
$\langle \delta_j(a),b\rangle=\langle a,\delta_j(b)\rangle$ for all $1\leq j \leq n$, we obtain
\[
l^{\a-\a'}\langle \Delta_k^{\b'} \tau(k),U_l\rangle = \langle l\rangle^{-2N}\sum_{|\mu|\leq 2N}c_{\mu,N} \langle\delta^{\mu+\a-\a'}\Delta_k^{\b'}
 \tau(k),U_l\rangle\, .
\]
This yields the following estimate:
\begin{align*}
\norm{\delta^\a \Delta^\b_k \rho^{(l)}(k)} \leq &\langle l\rangle^{-2N} \sum_{(\a',\b',\mu)\in F_{\a,\b,N}} \tbinom{\a}{\a'}\tbinom{\b}{\b'} c_{\mu,N} \, p_{\mu+\a-\a',\b'}^{({m'})}(\tau)\\
&p_{\a',\b-\b'}^{(m)}(\sg ) \, \langle k\rangle^{{m'}-|\b'|} \, \langle k+l+\b'\rangle^{m-|\b-\b'|}\, .
\end{align*}
where $F_{\a,\b,N}$ is the finite set $\set{(\a',\b',\mu)\in \N^{3n} \ : \ \a'\leq \a, \, \b'\leq \b\, , |\mu|\leq 2N}$.
Peetre's inequality: $\langle x+y\rangle^t\leq \sqrt 2^{\vert t\vert}\, \langle x\rangle^t\, \langle  y\rangle^{\vert t\vert}$, which holds for any real number $t$ and any $x,y$ in $\R^n$, yields
\[ 
\langle k+l+\b'\rangle^{m-|\b-\b'|}\leq (\sqrt{2}\langle \b'\rangle)^{\vert m-|\b-\b'|\vert}  \,\langle l\rangle^{\vert m-|\b-\b'|\vert}\, \langle k\rangle^{ m-|\b-\b'|} 
\] 
and hence 
\begin{equation}
\label{prodest}
p_{\a,\b}^{(m+{m'})}(\rho^{(l)})\leq \langle l\rangle^{-2N+|m|+|\b|} C_{\a,\b,N} \sum_{(\a',\b',\mu)\in F_{\a,\b,N}} p_{\mu+\a-\a',\b'}^{({m'})}(\tau)
p_{\a',\b-\b'}^{(m)}(\sg)\, .
\end{equation}
where $C_{\a,\b,N}:=\max_{(\a',\b',\mu)\in F_{\a,\b,N}} \tbinom{\a}{\a'}\tbinom{\b}{\b'} c_{\mu,N} (\sqrt{2}\langle \b'\rangle)^{|m-|\b-\b'||}$.
Choosing $N$ such that $-2N+|m|+|\b|<-n$ leads to the desired summability.
\end{proof}
\begin{ex}\label{ex:secprem} 
Let $\mathbf{\Delta}$ denote the operator $\sum_j \delta_j^2$ on $\A_\th$ as in Section \ref{secprem}. It follows from the above theorem, that the second order invertible pseudodifferential operator  $P:=1+\mathbf{\Delta}$ has powers given by $P^s(U_k)= \langle k\rangle^{2s} U_k$ for $s\in \R$.
\end{ex}
Note that the space $S_{\A_\th}(\Z^n)$, endowed with the star--product $\circ_\th$, is an $\R$--graded algebra, and $S^{-\infty}_{\A_\th}(\Z^n)$ is an ideal of this algebra.
Moreover, by \eqref{prodest}, the star--product is continuous as a bilinear map from $S_{\A_\th}^m(\Z^n)\times S_{\A_\th}^{m'}(\Z^n)$ into $S_{\A_\th}^{m+m'}(\Z^n)$. As a result, the composition of operators is continuous from $\Psi_\th^m(\T^n)\times \Psi_\th^{m'}(\T^n)$ into $\Psi_\th^{m+m'}(\T^n)$ with respect to the topology on $\Psi^{m+m'}_\th(\T^n)$ induced by that of $S^{m+m'}_{\A_\th}(\Z^n)$ via the isomorphism $\Op_\th$.

With the notation of Theorem \ref{thmcomp} we have for all $\sg,\tau\in S_{\A_\th}(\Z^n)$,
\begin{equation}
\label{bracketOp}
\left[\Op_\theta (\sigma),\Op_\theta(\tau)\right] = \Op_\theta\left(\{\sigma, \tau\}_\theta\right)
\end{equation}
where we have set $[A,B]:=AB-BA$, and $\{\sigma, \tau\}_\theta:= \sigma\circ_\theta \tau-\tau\circ_\theta\sigma$\index{$\{\sigma, \tau\}_\theta$} is called the \emph{star--bracket} (or simply commutator) of $\sigma$ and $\tau$.

Consider the derivation $\delta_j$ defined in \eqref{smalldeltaj}. We denote by $\bar \delta_j : S^m_{\A_\th}(\Z^n) \to S^m_{\A_\th}(\Z^n)$ the map defined as 
\begin{equation}
\label{bardelta}
\bar \delta_j (\sg)(k):=\delta_j(\sg(k)) \qquad \text{for all } k\in \Z^n\, . \index{$\bar \delta_j$}
\end{equation}
If $\a\in \N^n$ we denote by $\bar \delta^\a : S^m_{\A_\th}(\Z^n) \to S^m_{\A_\th}(\Z^n)$ the map defined as $\bar \delta^\a(\sg)(k):=\delta^\a(\sg(k))$ for all $k\in \Z^n$. The maps $\bar \delta^\a : S^m_{\A_\th}(\R^n) \to S^m_{\A_\th}(\R^n)$ are defined similarly.

\begin{ex} 
Let $\sg\in S^m_{\A_\th}(\Z^n)$. For all $j\in \{1, \cdots, n\}$,
\begin{align}
&\label{bracketkj}\{\sigma, k_j U_0\}_\theta= \bar\delta_j \sigma \, , \\
&\label{bracketUj}\{\sigma, U_{e_j}\}_\theta=  \Delta_j(\sigma)\,U_{e_j}+\sum_{l\in \Z^n} \sigma_l \, \left[U_l, U_{e_j}\right].
\end{align}
\end{ex} 

\begin{rk}
\label{rk:bar t a trace}
Note that the map $\ol\t$ given in Definition \ref{defn:bar t} vanishes on commutators of the pointwise algebra $(S_{\A_\th}(\Z^n),\cdot)$, but it does not vanish on commutators of the star--product algebra $(S_{\A_\th}(\Z^n),\circ_\th)$. 
\end{rk}

The Sobolev space $\H^s$ ($s\in \R$) associated to the noncommutative torus is defined as the Hilbert completion of $\A_\th$ for the following scalar product:
\[
\langle a,b\rangle_{s} := \sum_{k\in \Z^n} \langle k\rangle^{2s} a_k b_k\, ,
\]
where $a=\sum_{k\in \Z^n} a_k U_k$ and $b=\sum_{k\in \Z^n} b_k U_k$.

If $s=0$, the space $\H^0$ is the space $\H$ introduced in Section \ref{secprem}.

\begin{thm}
\label{thm:Sobolev}
(i) Any pseudodifferential operator of order $m$ is continuous from $\H^s$ into $\H^{s-m}$, for all $s\in \R$.

\medskip

(ii) Any pseudodifferential operator $A$ of order $m<-n$ is trace--class on $\H$. Moreover, 
$\Tr (A)= \sum_{\Z^n} \ol\t(\sg_A)$ where $\sg_A:=\Op_\th^{-1}(A)$ is the symbol of $A$. 
\end{thm}

\begin{proof}
$(i)$ The proof is similar to \cite[Proposition 4.2.3]{RT4}.

$(ii)$ Let $\mathbf{\Delta}$ denote the operator $\sum_j \delta_j^2$ on $\A_\th$ as in Section \ref{secprem}. With the notation of Example \ref{ex:secprem} setting  $P:=1+\mathbf{\Delta}$ we have  $P^s(U_k)= \langle k\rangle^{2s} U_k$ for any $s\in \R$. In particular, $P^{m/2}$ is trace--class on $\H$ for $m<-n$. Let $A$ be a pseudodifferential operator of order $m<-n$ and let us write $A= P^{m/2} P^{-m/2} A=:P^{m/2} B$. By $(i)$, $B$ is a bounded operator on $\H$, from which it follows that, like $P^{m/2}$, $A$ is trace--class on $\H$. Then  ${\Tr}(A)= \sum_{k\in \Z^n}\langle A(U_k), U_k\rangle= \sum_{\Z^n}\ol\t(\sg_A(k))$.
\end{proof}

\section{Classical toroidal symbols via extension maps}
\label{secextensionmap}

As in the commutative toroidal calculus we proceed to singling out a subclass of symbols and associated operators, namely the classical or (one--step) polyhomogeneous ones. In this section we use $\B$ to denote either $\A_\th$ or $\C$.

\subsection{Extended toroidal symbols}

We shall now use the extension of a toroidal symbol \cite[Section 6]{RT1}, \cite[Section 4.5]{RT4}, which is a key tool to transpose well--known concepts for symbols on $\R^n$ to toroidal symbols.

As before, for any fixed $k\in\Z^n$, let $T_k$ denote the translation on symbols $\sigma\mapsto \sigma(\cdot+k)$. For $\B=\C$, by \cite[Prop.\ 2.52]{P2}, given a symbol $\sigma$ of order $m$, the translated symbol $T_k\sigma$ is a symbol with the same order as $\sigma$ (see below Remark \ref{rk:transl}).

\begin{defn}
\label{def:extension}
Let $\sg\in S_\B(\Z^n)$. An \emph{extension} of $\sg$ is a symbol $\wt \sg$ in $S_\B(\R^n)$ such that $\wt \sg_{|\Z^n}=\sg$. 

We define an \emph{extension map} as a linear map $e:S_\B(\Z^n) \to S_\B(\R^n)$  

- which sends $S^m_\B(\Z^n)$ continuously into $S^m_\B(\R^n)$ for all $m\in \R$,

- such that $e(\sg)$ is an extension of $\sg$ for all $\sg$ in $S_{\B}(\Z^n)$, 

- which commutes with translations: $e\circ T_k= T_k \circ e$ for all $k\in \Z^n$.
 \end{defn}

\begin{defn}
\label{extensionmaps}
An extension map $e$ from $S_\C(\Z^n)$ into $S_\C(\R^n)$ is \emph{normalised} if for all $\sg \in S_\C^m(\Z^n)$ with $m<-n$,
\begin{equation}
\label{eq:normalised}\int_{\R^n} e( \sg) = \sum_{k\in \Z^n} \sg(k)\, .
\end{equation}
An extension map $e$ from $S_{\A_\th}(\Z^n)$ into $S_{\A_\th}(\R^n)$ is called $\A_\th$--\emph{compatible} if we have $e(a \sg b) = a e(\sg) b$ for all $a,b\in \A_\th$, where we identify $\A_\th$ with its image through the canonical injection $a\mapsto (k\mapsto a)$ from $\A_\th$ into $S^0_{\A_\th}(\R^n)$, or $S^0_{\A_\th}(\Z^n)$. 

An extension map $e$ from $S_{\A_\th}(\Z^n)$ into $S_{\A_\th}(\R^n)$ is called $\ol \t$--\emph{compatible} if $e\circ \iota_\th \circ \ol \t = \iota_\th \circ \ol \t \circ e$ (see Definition \ref{defn:bar t} for the definition of $\ol \t$).
\end{defn}

\begin{lem}
\label{lem:permute extensions and traces}
 If $e$ is a $\ol \t$--compatible extension map $S_{\A_\th}(\Z^n)\to S_{\A_\th}(\R^n)$, then $e_\C:=\ol \t\circ e\circ \iota_\th$\index{$e_\C$} is an extension map $S_{\C}(\Z^n)\to S_{\C}(\R^n)$ and we have $e_\C \circ \ol \t = \ol \t \circ e$, as well as $i_\th \circ e_\C= e\circ i_\th$. 
\end{lem}

\begin{proof}
 This follows straightforwardly from the definition of $\ol \t$--compatible extension map.
\end{proof}

\begin{defn}
 A $\ol \t$--compatible extension map $e$ from $S_{\A_\th}(\Z^n)$ into $S_{\A_\th}(\R^n)$ is \emph{normalised} if $e_\C$ is normalised.
\end{defn}

\begin{lem}
\label{rk:extension delta}
If $e:S_{\A_\th}(\Z^n)\to S_{\A_\th}(\R^n)$ is an $(\A_\th,\ol \t)$--compatible extension map, then for all $j=1,\ldots,n$
\[
 e \circ \ol \delta_j = \ol \delta_j \circ e,
\]
where $\ol\delta_j$ are the maps defined in \eqref{bardelta}, and hence, for all $\alpha\in\N^n$, $e \circ \ol \delta^\alpha = \ol \delta^\alpha \circ e$.
\end{lem}

\begin{proof}
 Let $\sg\in S_{\A_\th}^m(\Z^n)$ where $m\in \R$. The symbol $\sg$ can be written uniquely as $\sg=\sum_{l\in\Z^n}\sg_lU_l$ where $\sg_l\in S_\C^m(\Z^n)$ for all $l\in \Z$ (see Lemma \ref{lem:fouriercoeff}). Since $e$ is an $(\A_\th,\ol \t)$--compatible extension map, we have 
\[
 \left(e(\sg)\right)_l  =\ol\t(U_{-l}e(\sg))=\ol\t\circ e(U_{-l}\sg)=e_\C\circ\ol\t(U_{-l}\sg)=e_\C(\sigma_l) \text{ for all }l\in\Z^n.
\]
Therefore, since $e\circ i_\th=i_\th \circ e_\C $, we get $e(\sg_l U_0) = e_{\C}(\sg_l) U_0 = e(\sg)_l U_0$, and this, together with the continuity of $e$ and the $\A_\th$-compatibility of $e$, yields
\begin{align*} e\circ \ol \delta_j(\sg) 
  = \sum_{l\in\Z^n} e(\sg_l U_0) l_j U_l =\sum_{l\in\Z^n}  e(\sg)_l U_0 l_j U_l = \ol \delta_j (e(\sg))\, .
\end{align*}

\end{proof}

\begin{lem}
\label{lem:extensionmap}
(i) The set of normalised extension maps from the space $S_\C(\Z^n)$ into $S_\C(\R^n)$ is nonempty. 

\medskip

(ii) The set of $(\A_\th,\ol \t)$--compatible extension maps from the space $S_{\A_\th}(\Z^n)$ into $S_{\A_\th}(\R^n)$ is nonempty.

\medskip

(iii) The set of normalised $\ol \t$--compatible extension maps from the space $S_{\A_\th}(\Z^n)$ into $S_{\A_\th}(\R^n)$ is nonempty.

\medskip

(iv) If $\wt \sg$ and $\wt \sg'$ are two extensions of a given symbol $\sg$ in $S_\B(\Z^n)$, then $\wt \sg\sim \wt \sg'$. In particular, if $e,e'$ are two extension maps, then $e-e'$ maps $S_\B(\Z^n)$ into $S_\B^{-\infty}(\R^n)$.

\medskip

(v) For all $\sg,\tau\in S_\B(\Z^n)$ and for any extension map $e:S_\B(\Z^n)\to S_\B(\R^n)$, $e(\sg\tau)\sim e(\sg)e(\tau)$.
\end{lem}

\begin{proof}  
$(i,ii,iii)$ Let $\rho_1\in C^\infty(\R,[0,1])$ be an even function such that $\supp \rho_1\subset ]-1,1[$ and $\rho_1(x)+\rho_1(1-x)=1$ for all $x\in [0,1]$. Define $\rho:\R^n \to [0,1]$ such that $\rho(x)=\rho_1(x_1)\rho_1(x_2)\cdots\rho_1(x_n)$ for all $x=(x_1,\ldots,x_n)\in \R^n$. Note that $\rho\in \SS(\R^n)$, $\rho(0)=1$, and its Fourier transform $\wh \rho$ satisfies the crucial property $\wh \rho(k)=\delta_{k,0}$, where $\delta_{k,0}$ stands for the Kronecker delta function. Define $e:S_\B(\Z^n)\to \B^{\R^n}$ as
\begin{equation}
\label{eq:enormalised}
 e(\sg)(\xi) := \sum_{k\in \Z^n}  \wh \rho(\xi-k)\, \sg(k)\, .
\end{equation}
Following the same arguments of the proof of  (the ``only if" part of) \cite[Theorem 4.5.3]{RT4}, we see that $e$ is an extension map from $S_\B(\Z^n)$ into $S_\B(\R^n)$. Moreover, $e$ is a normalised extension map if $\B=\C$ since for all $\sg\in S_\C^m(\Z^n)$ with $m<-n$
\[
 \int_{\R^n} e(\sigma) = \sum_{k\in \Z^n} \left(\int_{\R^n} \wh \rho(\xi-k)\,d\xi\right) \sg(k) =\rho(0)\,\sum_{k\in \Z^n}\sigma(k)=\sum_{k\in \Z^n}\sigma(k),
\]
where we use Fubini's Theorem since the map $(\xi,k)\mapsto \wh \rho(\xi-k)\, \sg(k)$ belongs to $L^1(\R^n\times\Z^n,\C)$, because $\wh \rho$ is a Schwartz function and $\sg\in L^1(\Z^n,\C)$.\\
One easily also checks that $e$ is an $(\A_\th,\ol \t)$--compatible extension map if $\B=\A_\th$. Hence $e$ defines a normalised $\ol \t$--compatible extension map from the space $S_{\A_\th}(\Z^n)$ into $S_{\A_\th}(\R^n)$.\\
Let us now prove the continuity of the extension map $e:S_{\A_\th}(\Z^n)\to S_{\A_\th}(\R^n)$. Let $\sigma\in S_{\A_\th}(\Z^n)$. By definition, $e(\sg)(\xi):= \sum_{k\in \Z^n} \wh \rho(\xi-k)\, \sg(k)\,$. 
Thus, from \cite[Lemma 4.5.1]{RT4}, given a symbol $\sg$ and multiindices $\a,\b\in \N^n$ (see \eqref{Frechetseminorms}) we obtain:
\begin{align*}
\del_\xi^\b (\ol\delta^\a e(\sg)) (\xi) 
&= \sum_{k\in \Z^n} (\del_\xi^\b \wh \rho) (\xi-k)\, (\ol \delta^\a\sg)(k) \\
&= \sum_{k\in \Z^n} (\ol \Delta^\b \phi_\b)(\xi-k)\, (\ol \delta^\a\sg)(k)\\
&= (-1)^{|\b|}\sum_{k\in \Z^n} \phi_\b(\xi-k)\, (\Delta^\b \ol \delta^\a\sg)(k),
\end{align*}
where $\ol \Delta_j = I- T_{-e_j}$, $\ol \Delta^\b = \ol \Delta_1^{\b_1}\cdots \ol \Delta_n^{\b_n}$, and where the $\phi_\b$ are functions in $\SS(\R^n)$. Using the notation $\xi-\Z^n := \set{\xi-k \ :\ k\in \Z^n}$ and Peetre's inequality, 
the above computation implies that for all $\sg\in S^m_{\A_\th}(\Z^n)$, and all $\a,\b\in \N^n$,
\begin{align*}
 \norm{\del_\xi^\b (\ol\delta^\a e(\sg))(\xi)}
 &\leq \sum_{k\in \Z^n} | \phi_\b(\xi-k)|\, \norm{(\Delta^\b \ol \delta^\a(\sg))(k)}\\
 &\leq  p^{(m)}_{\a,\b}(\sg)  \sum_{k\in \Z^n} | \phi_\b(\xi-k)|\,\langle k\rangle^{m-|\b|}\\
 & \leq p^{(m)}_{\a,\b}(\sg)  \sum_{\eta\in \xi-\Z^n} | \phi_\b(\eta)|\,\langle \xi -\eta\rangle^{m-|\b|}\\
 & \leq  \langle \xi\rangle^{m-|\b|} p^{(m)}_{\a,\b}(\sg) 2^{|m-|\b||} \sum_{\eta\in \xi-\Z^n} | \phi_\b(\eta)|\, \langle \eta\rangle^{|m-|\b||} \\
 &\leq \langle \xi\rangle^{m-|\b|} C_{\b,m} p^{(m)}_{\a,\b}(\sg), 
\end{align*}
where $C_{\b,m}:= \sup_{\xi\in \R^n} g_{\b,m}(\xi)$, and $g_{\b,m}$ is the bounded ($\Z^n$--periodic) function $\xi\mapsto 2^{|m-|\b||} \sum_{\eta\in \xi-\Z^n} |\phi_\b(\eta)|\, \langle \eta\rangle^{|m-|\b||}$.
This yields the following estimate 
\[
p^{(m)}_{\a,\b}(e(\sg)) \leq C_{\b,m} p^{(m)}_{\a,\b} (\sg),
\]
from which we deduce the continuity of the extension map $e$ for the Fr\'echet topologies of symbols spaces.

\medskip

$(iv)$ This follows from a straightforward modification of the proof for the commutative case \cite[Theorem 4.5.3]{RT4}.

\medskip

$(v)$ This follows from $(iv)$ since $e(\sg\tau)_{|\Z^n}=\sg\tau$.
\end{proof}

\begin{defn}
 A symbol $\tau\in S_\B(\R^n)$ is \emph{positively homogeneous of degree} $m\in \C$ if $\tau\in S_\B^{\Real(m)}(\R^n)$ and $\tau(t \xi) = t^m \tau(\xi)$ for all $t>1$ and $|\xi|\geq 1$. We will denote by $HS^m_{\B}(\R^n)$\index{$HS^m_{\B}(\R^n)$} the space of all positively homogeneous symbols of degree $m$ in $S^{\Real(m)}_{\B}(\R^n)$.
\end{defn}

The following fact will be used later in the crucial Lemma \ref{lem:translationinv}:

\begin{lem}
\label{closedness} 
Let $m\in \C$. The space $HS^m_\B(\R^n)$ is a closed subspace of the Fr\'echet space $S^{\Real(m)}_\B(\R^n)$.
\end{lem}

\begin{proof}
Define for all $t>1$, $L_t : \sg \mapsto \sg(t\cdot) -t^m \sg$. It is easy to check that $L_t$ is a continuous linear operator from $S^{\Real(m)}_\B(\R^n)$ into itself. By definition, $HS^m_\B(\R^n)= \cap_{t>1} L_t^{-1}(C^\infty_{B})$ where $C^\infty_{B}$ denotes the space of all smooth functions $\R^n\to \B$ that are zero outside the open unit ball. We have $C^\infty_{B}=\cap_{|\xi|\geq 1} ({\rm ev}_\xi\circ \iota)^{-1}(0)$, where $\iota$ is the canonical continuous inclusion of $S^{\Real(m)}_\B(\R^n)$ into $C^\infty(\R^n,\B)$ and ${\rm ev}_\xi$ is the continuous linear map $\sg\mapsto \sg(\xi)$ from $C^\infty(\R^n,\B)$ into $\B$. Thus, $C^\infty_B$ is closed in $S^{\Real(m)}_\B(\R^n)$, and the result follows.
\end{proof}

\begin{defn}
A symbol $\sg \in S_{\B}(\R^n)$ is called \emph{positively quasihomogeneous symbol of degree $m\in \C$} if there exists a positively homogeneous symbol $\tau$ of degree $m$ such that $\tau\sim \sg$.  We will denote by $QS^m_{\B}(\R^n)$\index{$QS^m_{\B}(\R^n)$} the space of all positively quasihomogeneous symbols of degree $m$.
\end{defn}

\begin{prop}
\label{prop:quasihom}
 Let $m\in \C$ and $\sg\in S^{\Real(m)}_{\B}(\Z^n)$. The following are equivalent:
 
(i) There exists an extension map  $e: S_{\B}(\Z^n)\to S_{\B}(\R^n)$, such that $e(\sg)\in QS^m_{\B}(\R^n)$.

(ii) For any extension map $e: S_{\B}(\Z^n)\to S_{\B}(\R^n)$, $e(\sg)\in QS^m_{\B}(\R^n)$.

If one of these conditions is satisfied, we say that $\sg$ is \emph{positively quasihomogeneous of degree} $m$, and we write $\sg\in QS^m_\B(\Z^n)$. 
\end{prop}

\begin{proof}
This follows directly from Lemma \ref{lem:extensionmap}.
\end{proof}

\subsection{The algebra of noncommutative toroidal classical symbols}

With the definitions of homogeneity and quasihomogeneity we construct the definition of toroidal classical symbols:

\begin{defn}
 Let $m\in \C$. A symbol $\sg\in S_\B(\R^n)$ is \emph{classical (or one--step polyhomogeneous) of order} $m$ if there exists a sequence $(\sg_{[m-j]})_{j\in \N}$\index{$\sg_{[m-j]}$} satisfying $\sg_{[m-j]}\in HS^{m-j}_\B(\R^n)$ for all $j\in \N$, and such that $\sg\sim \sum_j \sg_{[m-j]}$. Equivalently, we can replace homogeneous by quasihomogeneous in this definition. We denote by $CS_\B^m(\R^n)$\index{$CS_\B^m(\R^n)$} the space of all classical symbols of order $m$, $CS_\B(\R^n)$\index{$CS_\B(\R^n)$} the set of all classical symbols, $CS_\B^{\Z}(\R^n)$\index{$CS_\B^{\Z}(\R^n)$} the space of all classical symbols of integer order, and by $CS_\B^{\notin \Z}(\R^n)$\index{$CS_\B^{\notin \Z}(\R^n)$} the set of all classical symbols of non--integer order.
\end{defn}

Recall that if $\sg\in CS_\B^m(\R^n)$ then there is a unique sequence $([\sg]_{m-j})_j$\index{$[\sg]_{m-j}$} such that $[\sg]_{m-j}$ is an equivalence class (modulo smoothing symbols) of a positively quasihomogeneous symbol of degree $m-j$, and $\sg\sim \sum \sg_{[m-j]}$ for all sequences $(\sg_{[m-j]})_j$ such that $\sg_{[m-j]}\in [\sg]_{m-j}$ for all $j\in\N$. \\

We now extend these usual definitions to the case of discrete symbols:

\begin{defn}
\label{defclass}
 Let $m\in \C$. A symbol $\sg\in S_{\B}(\Z^n)$ is \emph{classical (or one--step polyhomogeneous) of order} $m$ if there exists a sequence $(\sg_{[m-j]})_{j\in \N}$ satisfying $\sg_{[m-j]}\in QS^{m-j}_{\B}(\Z^n)$ for all $j\in \N$, and such that $\sg\sim \sum_j \sg_{[m-j]}$ (see Proposition \ref{prop:quasihom}). Such a sequence will be called \emph{positively quasihomogeneous resolution} of $\sg$. If $\sg_{[m-j]}\in HS^{m-j}_{\B}(\Z^n)$ for all $j\in \N$ the sequence $(\sg_{[m-j]})_j$ will be called \emph{positively homogeneous resolution} of $\sg$.\\
 We denote by $CS^{m}_{\B}(\Z^n)$\index{$CS^{m}_{\B}(\Z^n)$} the space of all classical symbols of order $m$ and by $CS_{\B}(\Z^n):=\cup_{m\in \C} CS_{\B}^m(\Z^n)$\index{$CS_{\B}(\Z^n)$} the set of all classical symbols. The set $CS^{\Z}_{\B}(\Z^n):=\cup_{m\in \Z} CS^m_{\B}(\Z^n)$\index{$CS^{\Z}_{\B}(\Z^n)$} of all classical symbols of integer order forms a subalgebra of the algebra of linear combinations of elements of the monoid $CS_{\B}(\Z^n)$.\\
 We shall use for convenience the notation $CS^{<-n}_{\B}(\Z^n):= \cup_{\Real(m)<-n}CS_{\B}^m(\Z^n)$\index{$CS^{<-n}_{\B}(\Z^n)$} and denote by $CS_{\B}^{\notin \Z}(\Z^n)$\index{$CS_{\B}^{\notin \Z}(\Z^n)$} the set $CS_{\B}(\Z^n)\backslash CS_{\B}^{\Z}(\Z^n)$.
\end{defn}

\begin{lem}
 \label{techhom}
   (i) Let $\sg,\sg'\in QS^m_{\B}(\Z^n)$ be such that $\sg-\sg'\in S^{\Real(m)-1}_{\B}(\Z^n)$. Then $\sg\sim\sg'$. 
  
  \medskip
   
   (ii) Let $\sg\in CS^m_{\B}(\Z^n)$. Then there exists a unique sequence of $\sim$--equivalence classes $([\sg]_{m-j})_{j\in \N}$ with $[\sg]_{m-j}\in QS^{m-j}_{\B}/\sim$ such that $\sg\sim \sum_{j} \sg_{[m-j]}$ for any sequence $(\sg_{[m-j]})_{j\in \N}$ with $\sg_{[m-j]}\in [\sg]_{m-j}$.
    
   If $\sg\in CS_\B(\Z^n)$, and $s\in \C$, we set $[\sg]_s$ to be $[\sg]_{m-j}$ when there is $j\in \N$ and $m\in \C$ such that $\sg\in CS_\B^m(\Z^n)$ and $s=m-j$, and zero otherwise. 
\end{lem}
 
\begin{proof}
   $(i)$ Define $\tau:=\sg-\sg'\in S^{\Real(m)-1}_{\B} (\Z^n) \cap QS^m_{\B}(\Z^n) $. Let $e$ be an extension map $S_{\B}(\Z^n)\to S_{\B}(\R^n)$. It follows from Proposition \ref{prop:quasihom} that $e(\tau)\in S^{\Real(m)-1}_{\B} (\R^n) \cap QS^m_{\B}(\R^n)$. Thus, there is $\tau'\in S^{\Real(m)-1}_{\B}(\R^n)$ such that $e(\tau)\sim \tau'$ and $\tau'(t\xi)=t^m \tau'(\xi)$ for all $t>1$ and $|\xi|\geq 1$. Moreover, there is $C\in \R$ such that for all $\xi\in \R^n$, $\norm{\tau'(\xi)}\leq C \langle \xi\rangle^{\Real(m)-1}$. As a consequence, we obtain for all $t>1$ and for all $\xi\in\R^n\backslash B(0,1)$, ($B(0,1)$ is the ball with center 0 and radius 1),
   \[
   \norm{\tau'(\xi)} \leq C t^{-\Real(m)}\langle t\xi\rangle^{\Real(m)-1} = C t^{-1} \, (1/t^2+|\xi|^2)^{(\Real(m)-1)/2}
   \]
   which implies that $\tau'(\xi)=0$ when $|\xi|\geq 1$, and in particular that $\tau'$ is compactly supported. As a consequence, $\tau'$ and therefore $e(\tau)$, belong to $S^{-\infty}_{\B}(\R^n)$. Lemma \ref{restr} now yields that $\tau\in S^{-\infty}_{\B}(\Z^n)$.
 
 \medskip
 
  $(ii)$ The existence is clear by definition of $CS_{\A_\th}^m(\Z^n)$. To prove uniqueness, suppose that $(c_{m-j})_j$ and $(c'_{m-j})_j$ are two such sequences, and let $(\sg_{[m-j]})_j$ (resp.\ $(\sg'_{[m-j]})_j$) be a sequence such that $\sg_{[m-j]}\in c_{m-j}$ (resp.\ $\sg'_{[m-j]}\in c'_{m-j}$) for all $j$.
  If we prove that $\sg_{[m-j]}\sim \sg'_{[m-j]}$ for all $j\in \N$, we are done. Let us check this for $j=0$. We have $\sg-\sg_{[m]}$ and $\sg-\sg'_{[m]}$ belong to $S^{\Real(m)-1}_{\B}(\Z^n)$. Thus, $\sg_{[m]}-\sg'_{[m]}\in S^{\Real(m)-1}_{\B}(\Z^n)$. Since $\sg_{[m]}-\sg'_{[m]}\in QS^m_{\B}(\Z^n)$, $(i)$ implies that $\sg_{[m]}\sim \sg'_{[m]}$. Suppose now that $\sg_{[m-j]}\sim \sg'_{[m-j]}$ for all $j\in \set{0,\cdots,p}$ for some $p\in \N$. We have $\sum_{j=0}^{p+1} \sg_{[m-j]}-\sg'_{[m-j]} \in S^{m-(p+2)}_{\B}(\Z^n)$, which implies by induction hypothesis, that $\sg_{[m-(p+1)]}-\sg'_{[m-(p+1)]}\in S^{m-(p+2)}_{\B}(\Z^n)$. Thus, $(i)$ implies that $\sg_{[m-(p+1)]}\sim \sg'_{[m-(p+1)]}$.
\end{proof}
 
\begin{prop}
 \label{prop:classical}
 Let $m\in \C$. The following are equivalent:
 
 (i) $\sg\in CS^m_{\B}(\Z^n)$.
 
 (ii) There exists an extension map $e:S_{\B}(\Z^n)\to S_\B(\R^n)$ such that $e(\sg)\in CS_\B^m(\R^n)$.
 
 (iii) For any extension map $e:S_{\B}(\Z^n)\to S_\B(\R^n)$, $e(\sg)\in CS_\B^m(\R^n)$. 
 
Moreover, if $\sg\in CS^m_\B(\Z^n)$, and $e:S_{\B}(\Z^n)\to S_{\B}(\R^n)$ is an extension map, then $e(\sg_{[m-j]}) \sim (e(\sg))_{[m-j]}$ for all $\sg_{[m-j]}$ in the equivalence class $[\sg]_{m-j}$ and $(e(\sg))_{[m-j]}$ in the equivalence class $[e(\sg)]_{m-j}$. 
In other words, taking extensions and taking quasihomogeneous parts are commuting operations modulo smoothing terms.
\end{prop}
 
\begin{proof}
 $(i)\Rightarrow (iii)$ Suppose that $\sg$ is in $CS_{\B}(\Z^n)$ and let $e$ be an extension map. Let $(\sg_{[m-j]})_j$ be a sequence such that $\sg\sim \sum_j \sg_{[m-j]}$ and $\sg_{[m-j]}\in QS^{m-j}_{\B}(\Z^n)$. We have for all $j\in \N$, $e(\sg)-e(\sum_{i=0}^j \sg_{[m-i]}) \in S_\B^{\Real(m)-j-1}(\R^n)$. Thus since $e(\sg_{[m-i]})\in QS^{m-i}_\B(\R^n)$ for all $i=0,\ldots,j$, we obtain that $e(\sg)\in CS_\B^m(\R^n)$.
 
 $(iii)\Rightarrow (ii)$ This is straightforward.
 
 $(ii)\Rightarrow (i)$ Suppose that there is an extension map $e$ such that $e(\sg)\in CS_\B^m(\R^n)$. Let $(\sg_{[m-j]})_j$ be a sequence 
 such that $e(\sg)\sim \sum_j \sg_{[m-j]}$ and $\sg_{[m-j]}\in QS^{m-j}_{\B}(\R^n)$. Lemma \ref{restr} implies that $\sg-\sum_{i=0}^j (\sg_{[m-i]})_{|\Z^n}\in S^{\Real(m)-j-1}_{\B}(\Z^n)$ for all $j\in \N$. 
 Since $e((\sg_{[m-i]})_{|\Z^n})-\sg_{[m-i]}\in S_\B^{-\infty}(\R^n)$ by Lemma \ref{lem:extensionmap} $(iii)$, it follows that 
 $e((\sg_{[m-i]})_{|\Z^n})\in QS^{m-i}_{\B}(\R^n)$, and thus $(\sg_{[m-i]})_{|\Z^n}\in QS^{m-i}_\B(\Z^n)$.
 This yields the result.
  
 The last statement follows from Lemma \ref{techhom}.
\end{proof}
 
\begin{rk}
 \label{rk:transl}
 Note that $CS^m_{\B}(\Z^n)$ is stable under the $\Z^n$--translations $T_l$. Indeed, if $e:S_{\B}(\Z^n)\to S_{\B}(\R^n)$ is an extension map, then a Taylor expansion shows that $T_{l}$ maps $CS^m_{\B}(\R^n)$ into $CS^{m}_{\B}(\R^n)$. Since $T_{l} \circ e = e\circ T_{l}$, it follows that $T_{l}$ maps $CS^m_{\B}(\Z^n)$ into $CS^{m}_{\B}(\Z^n)$.
 Similarly, note by Remark \ref{rk:extension delta}, that the space $CS^m_{\A_\th}(\Z^n)$ is stable under the $\ol \delta_j$ maps given in \eqref{bardelta}.
\end{rk}
  
\begin{lem}
 \label{lem:permut}
 (i) $\ol \t $ maps $CS_{\A_\th}^m(\Z^n)$ (resp.\ $QS_{\A_\th}^m(\Z^n)$) into $CS_\C^m(\Z^n)$ (resp.\ $QS_{\C}^m(\Z^n)$) for all $m\in \C$.
 Similarly, this holds for spaces of smooth classical symbols.
 
 \medskip
 
 (ii) For all $m\in \C$ and $\sg\in CS^m_{\A_\th}(\Z^n)$, we have $\ol \t(\sg_{[m-j]}) \sim (\ol \t(\sg))_{[m-j]}$ for all $\sg_{[m-j]}\in [\sg]_{m-j}$ and $(\ol \t(\sg))_{[m-j]}\in [\ol \t(\sg)]_{m-j}$. The same property holds for a smooth symbol $\sg\in CS^m_{\A_\th}(\R^n)$.
 In other words, taking pointwise traces and taking quasihomogeneous parts are commuting operations modulo smoothing terms. 
\end{lem}

\begin{proof}
 $(i)$ We first check the case of smooth classical symbols. Let $\sg\in CS_{\A_\th}^m(\R^n)$. Let $(\sg_{[m-j]})_j$ be a sequence such that $\sg\sim \sum_j \sg_{[m-j]}$ and $\sg_{[m-j]}\in QS^{m-j}_{\A_\th}(\R^n)$. We have for all $j\in \N$,  $\ol \t(\sg)-\ol \t(\sum_{i=0}^j \sg_{[m-i]}) \in S_\C^{\Real(m)-j-1}(\R^n)$. Thus, it is enough to check that $\ol \t$ maps $QS^m_{\A_\th}(\R^n)$ into $QS^m_{\C}(\R^n)$. This follows from the linearity of $\t$. The case of discrete symbols follows from $(i)$ and from the case of smooth symbols.
 
 \medskip
 
 $(ii)$ This follows directly from $(i)$.
 \end{proof}

\begin{thm}
 \label{thm:asympt} 
 (i) If $\sg,\tau\in S_{\A_\th}(\Z^n)$, then for any extension map $e$
 \[
  \sg\circ_\th\tau \sim \sum_{\a\in \N^n} \tfrac{1}{\a !} (\del_\xi^\a e(\sg))_{|\Z^n} \, \bar\delta^{\a}\tau,
 \]
 and
 \[
 e(\sg\circ_\th \tau) \sim \sum_{\a\in \N^n} \tfrac{1}{\a !} (\del_\xi^\a e(\sg)) \, \bar\delta^{\a} e(\tau) \, .
 \]
 
 (ii) Let $\sg$ be a symbol in $CS^m_{\A_\th}(\Z^n)$ and $\tau\in CS_{\A_\th}^{m'}(\Z^n)$, and let $(e(\sg)_{[m-j]})_j$, 
 $(e(\tau)_{[m'-j]})_j$ be positively homogeneous resolutions of respectively $e(\sg)$ and $e(\tau)$, where $e$ is an extension map. 
 Then 
 \[
 e(\sg\circ_\th \tau) \sim \sum_j (\sg\circ_\th\tau)^e_{[m+m'-j]}
 \] 
 where
 \[
 (\sg\circ_\th\tau)^e_{[m+m'-j]} := \sum_{|\a|+i+i'=j} \tfrac{1}{\a!} (\del_\xi^\a e(\sg)_{[m-i]})\, \ol\delta^\a e(\tau)_{[m'-i']} \,   
 \]
 belongs to $HS^{m+m'-j}_{\A_\th}(\R^n)$. In particular, the star--product $\circ_\theta$ of toroidal symbols maps $CS_{\A_\th}^{m}(\Z^n)\times CS_{\A_\th}^{m'}(\Z^n)$ into $CS_{\A_\th}^{m+m'}(\Z^n)$. Thus, the set $CS_{\A_\th}(\Z^n)=\cup_{m\in \C} CS_{\A_\th}^m(\Z^n)$ is a monoid under the star--product $\circ_\th$.
 Note that this is not an algebra, as it is not stable under addition.
\end{thm}

\begin{proof}
$(i)$ Without loss of generality we can assume that the extension map $e$ is an $(\A_\th,\ol \t)$--compatible extension map. From Theorem \ref{thmcomp} and the fact that $T_l\sg= (T_l e(\sg))_{|\Z^n}$, a Taylor expansion of $T_le(\sg)$ allows to deduce that for all $N\in \N$,
\[
 \sg\circ_\th \tau = \sum_{l\in \Z^n}  \sum_{|\a|\leq N} \tfrac{l^\a}{\a!}(\del^\a_\xi e(\sg))_{|\Z^n}\tau_l U_l +R_{N,l}^\sg  
\]
where $R_{N,l}^\sg:= \left(\sum_{|\a|=N+1}\tfrac{N+1}{\a!} l^\a \int_0^1 (1-t)^N \del_\xi^\a e(\sg)(\cdot+tl)dt\, \tau_l U_l\right)_{|\Z^n}$.
The absolute summability of $(\wt R_l)_l$ in $S^{m+m'-N-1}_{\A_\th}(\R^n)$, where 
\[
\wt R_l:= \sum_{|\a|=N+1}\tfrac{N+1}{\a!} l^\a \int_0^1 (1-t)^N \del_\xi^\a e(\sg)(\cdot+tl)dt\, (e(\tau))_l U_l,
\]
follows from an application of Leibniz formula, Peetre's inequality and Lemma \ref{lem:fouriercoeff}. This implies that $\sg\circ_\th\tau \sim \sum_{\a\in \N^n} \tfrac{1}{\a !} (\del_\xi^\a e(\sg))_{|\Z^n} \, \bar\delta^{\a}\tau $. Applying now the extension map $e$ yields $e(\sg\circ_\th \tau) \sim \sum_{\a\in \N^n}e(\tfrac{1}{\a !} (\del_\xi^\a e(\sg))_{|\Z^n} \, \ol\delta^{\a}\tau)$. Since $e(\sg\tau) \sim e(\sg)e(\tau)$ (Lemma \ref{lem:extensionmap} $(iv)$) and $e(\ol \delta^\a \tau)= \ol\delta^\a e(\tau)$ (Remark \ref{rk:extension delta}), we get the result. 

$(ii)$ This follows directly from $(i)$.
\end{proof}

\begin{rk}
\label{rk:leadingsymbol}
From Theorem \ref{thm:asympt} it follows the multiplicativity of the leading symbol map, which corresponds to taking the homogeneous part of highest homogeneity degree in a positively homogeneous resolution of the symbol.
\end{rk}

As a direct consequence of Theorem \ref{thm:asympt}, we obtain:

\begin{cor}
\label{cor:bracketofsymbols}
 Let $\sg\in CS^m_{\A_\th}(\Z^n)$ and $\tau\in CS_{\A_\th}^{m'}(\Z^n)$ be two symbols, and let $(e(\sg)_{[m-j]})_j$, $(e(\tau)_{[m'-j]})_j$ be positively homogeneous resolutions of respectively $e(\sg)$ and $e(\tau)$, where $e$ is an extension map.
 Then the star--bracket $\set{\sg,\tau}_\th $ lies in $ CS^{m+m'}_{\A_\th}(\Z^n)$, and 
 \begin{align}
 \label{eqasympt}
 e(\set{\sg,\tau}_\th) \sim \sum_j \sum_{|\a|+i+i'=j} \frac{1}{\a!} \Big(&(\del_\xi^\a e(\sg)_{[m-i]}) \, \ol\delta^\a e(\tau)_{[m'-i']} \nonumber\\
 &-(\del_\xi^\a e(\tau)_{[m'-i']})\, \ol\delta^\a e(\sg)_{[m -i]}\Big)\, .
 \end{align} 
\end{cor}

\begin{rk}
\label{rk:order star bracket}
Note that in contrast with scalar valued symbols $\sigma$ in $ CS^{m }_{\C}(\Z^n)$ and $\tau$ in $ CS^{m' }_{\C}(\Z^n)$ for which the star--bracket $\set{\sg,\tau}$ lies in $CS^{m+m'-1}_\C(\Z^n)$, in the noncommutative setup one should expect $\set{\sg,\tau}_\th $ not to lie in the space $ CS^{m+m'-1}_{\A_\th}(\Z^n)$ for $\sg\in CS^m_{\A_\th}(\Z^n)$ and $\tau\in CS_{\A_\th}^{m'}(\Z^n)$.
\end{rk}

\begin{defn}
 The space of all \emph{classical pseudodifferential operators of order $m$ on $\T^n$} is the set $\Cl_\theta^m(\T^n):= {\Op}_\theta(CS_{\A_\th}^m(\Z^n))$,\index{$\Cl_\theta^m(\T^n)$} and we define $\Cl_\th^{\Z}(\T^n)$\index{$\Cl_\th^{\Z}(\T^n)$} as well as $\Cl_\th^{\notin \Z}(\T^n)=\Cl_\th(\T^n)\backslash \Cl_\th^{\Z}(\T^n)$\index{$\Cl_\th^{\notin \Z}(\T^n)$} similarly.
\end{defn}

\section{Traces and translation invariant linear forms}
\label{sectraces}

As in the previous section we use $\B$ to denote either $\A_\th$ or $\C$, and unless otherwise specified, $m$ denotes a complex number.

\subsection{Linear forms on toroidal symbols}
\label{subsec:linear forms on symbols}

We call a functional $\lambda$ on a subset $S$ of a vector space $V$ into $\C$ a \emph{linear form} if for any $v_1,v_2\in S$ and $\alpha_1,\alpha_2\in \C$, $\lambda\left(\alpha_1 v_1+\alpha_2 v_2\right) = \alpha_1 \lambda(v_1)+\alpha_2 \lambda(v_2)$ whenever $\alpha_1 v_1+\alpha_2 v_2\in S$.

\begin{defn}
\label{exotic}
A linear form on a subset $S\subseteq S_{\B}(\Z^n)$ is \emph{exotic} (resp.\ \emph{singular}) if it vanishes on symbols whose order has real part $<-n$ (resp.\ on smoothing symbols). A similar definition holds with $\Z^n$ replaced by $\R^n$.\\
A linear form $\la$ on a subset $S\subset \Cl_\th(\T^n)$ is \emph{exotic} (resp.\ \emph{singular}) if the corresponding linear form $\la\circ \Op_\th$ on $\Op_\th^{-1}(S)$ is exotic (resp.\ singular), or equivalently if $\lambda$ vanishes on operators in $S$ whose order has real part $<-n$ (resp.\ on smoothing operators).  
\end{defn}

\begin{rk} 
 Note that symbols (resp.\ operators)  whose order has real part $<-n$, are in $\ell^1(\Z^n,\A_\th)$ (resp.\ trace--class on $\H$ by Proposition \ref{thm:Sobolev} $(ii)$), so that an exotic trace vanishes on $\ell^1(\Z^n,\A_\th)$--symbols (resp.\ trace--class operators). 
\end{rk}

\begin{rk}
 The terminology ``exotic" is borrowed from \cite{Sc}, whereas the terminology ``singular" is widespread in the literature on pseudodifferential operators. Also, Wodzicki used the term exotic for determinants defined by means of a residue (see \cite{K}, \cite{W1})  which with our terminology, is an exotic trace. Clearly, exotic linear forms are singular but a singular trace need not be exotic, as we shall see later (Remark \ref{rk:leading symbol singular not exotic}) with leading symbol traces on certain trace--class operators; see also \cite{AGPS} where the existence of a trace which vanishes on finite rank operators but not on all trace--class operators is shown.
\end{rk}

\begin{defn}
\label{def:trace}
 Let $T,S$ be subsets of an algebra. A linear form on $S$ is called a $T$--\emph{trace} on $S$ if it vanishes on commutators of the form $[A,B]:=AB-BA$ where $A,B\in T$ and $[A,B]\in S$. If $S=T$ the linear form is called a \emph{trace} on $T$. 
\end{defn}

Unless otherwise specified, a \emph{trace} on classical symbols is understood in the sense of a $CS_{\A_\th}(\R^n)$--, resp.\ $CS_{\A_\th}(\Z^n)$--trace (here the commutator is defined with $\circ_\th$), and similarly for traces on classical pseudodifferential operators. We shall in particular consider $CS_{\A_\th}(\R^n)$--, resp.\ $CS_{\A_\th}(\Z^n)$--traces on $CS_{\A_\th}^m(\Z^n)$ (resp.\ $CS_{\A_\th}^m(\R^n)$) for a fixed complex order $m$.
  
\begin{defn}
\label{ell 1-continuous}
An $\ell^1$--\emph{continuous} linear form on a subset $S\subseteq CS_{\B}(\Z^n)$ (resp.\ $S\subseteq CS_{\B}(\R^n)$) is a linear form such that $\la_{|S\cap CS_{\B}^m(\Z^n)}$ (resp.\ $\la_{|S\cap CS_{\B}^m(\R^n)}$) is continuous for the $\ell^1(\Z^n,\B)$ (resp.\ $\ell^1(\R^n,\B)$) topology whenever $\Real(m)<-n$. 

We say that a linear form on a subset $S\subseteq \Cl_{\th}(\T^n)$ is $\mathcal{L}^1$--\emph{continuous} if $\la \circ \Op_\th$ is  $\ell^1$--continuous on $\Op_\th^{-1}(S)$.
\end{defn}

Let $\la$ be a linear form on $CS_{\B}(\Z^n)$ and let $s$ be a complex number. The map $\sg\mapsto \la(\sg_{[s]})$ which assigns the value $0$ to smoothing symbols, and assigns the value $\la(\sg_{[s]})$ to non--smoothing symbols $\sg$, where $\sg_{[s]}$ is an element of $[\sg]_{s}$, is well defined since $\la(\sg_{[s]})$ does not depend on the choice of $\sg_{[s]}$ in $[\sg]_s$. 
It is a singular linear form by construction and it is exotic whenever $\Real(s)\geq-n$.\\

Let us recall some useful linear forms on sets of classical symbols on $\R^n$ (see e.g.\ \cite{P1,P2}). We set $\Z_n:=\Z\cap [-n,+\infty[$ and denote by $CS_\C^{\notin \Z_n}(\R^n)$ the set of all classical symbols whose order lies in $\C\backslash \Z_n$.

\begin{defn}
 $(i)$ The \emph{noncommutative residue} $\res$ on $CS_\C(\R^n)$, defined as 
 \begin{equation}
 \label{eq:ncr}
 \res(\sg):=\int_{\Sbb^{n-1}} \sg_{[-n]} (\xi) \, dS(\xi)\, , \index{$\res$}
 \end{equation}
 where $dS$ is the volume form on ${\Sbb^{n-1}}$ induced by the canonical volume form on $\R^n$, is an exotic and hence singular linear form.\\

 $(ii)$ The \emph{cut--off integral} on $CS_\C(\R^n)$, defined as the linear form
 \begin{equation}
 \label{eq:cutoff}
 \cutoffint^{\cutoff}_{\R^n}\sg:=\underset{R\to \infty}{\fp} \int_{B(0,R)} \sg\, , \index{$\cutoffint^{\cutoff}_{\R^n}$}
 \end{equation} 
 coincides with the usual Lebesgue integral on $CS_\C^{<-n}(\R^n)$ and is $\ell^1$--continuous. Here $\underset{R\to \infty}{\fp} f$ stands for a (Hadamard) finite part of the expansion of $f$ as $R\to \infty$.\footnote{More precisely, if $f(R)\sim_{R\to \infty} \sum_{j=0}^\infty a_{j} R^{\alpha-j}$ for some complex numbers $\alpha$ and $a_{j}$ with $j$ in $\Z_{\geq 0}$, by which we mean $f(R)-\sum_{j=0}^{N-1} a_{j} R^{\alpha-j}= O\left(R^{{\rm Re}(\alpha)-N+\e}\right)$ for any positive $\e$ and any positive integer $N$, then $\underset{R\to \infty}{\fp}f$ is given by $ a_\alpha$ if $\alpha\in \Z_{\geq 0}$ and it vanishes otherwise.} \\ 
 When restricted to the set $CS_\C^{\notin \Z_n}(\R^n)$, the cut--off integral $\cutoffint^{\cutoff}_{\R^n}$ does not depend on a rescaling  $R\to t\, R$ for any positive $t$ \cite[Exercise 3.22]{P2}, and is called the \emph{canonical integral}, denoted by $\cutoffint_{\R^n}$\index{$\cutoffint_{\R^n}$}.\\

 $(iii)$ The \emph{cut--off discrete sum} on $CS_\C(\R^n)$, defined in \cite[Definition 5.26]{P2} as the finite part of the discrete sum on integer points of an expanded polytope $N\Delta$:
 \begin{equation}
 \label{eq:cutoffsum} 
 {\textstyle\cutoffsum_{\Z^n}^{\mathrm{cut-off}}}\sg :=\underset{N\to \infty}{\fp}\sum_{N\Delta\cap  \Z^n}\sg\, , \index{${\textstyle\cutoffsum_{\Z^n}^{\mathrm{cut-off}}}$}
 \end{equation}
 coincides with the usual discrete sum $\sum_{\Z^n}$ on $CS_\C^{<-n}(\R^n)$.\\
 When restricted to the set $CS_\C^{\notin \Z_n}(\R^n)$, the cut--off discrete sum  $\cutoffsum^{\cutoff}_{\Z^n}$ does not depend on a rescaling $N\to t\, N$ for any positive $t$, nor does it depend on the choice of polytope $\Delta$ \cite[Theorem 5.28]{P2}, and is called the \emph{canonical discrete sum}, denoted by $\cutoffsum_{\Z^n}$.\index{$\cutoffsum_{\Z^n}$}\\

 $(iv)$ A \emph{leading symbol form} on classical symbols on $\R^n$ of order $m\in \C$ is a map
 \begin{align*}
  CS_{\C}^m(\R^n)&\longrightarrow\C \\
  \sigma\sim \sum_j \sigma_{[m-j]}&\longmapsto L(\sigma_{[m]}),  
 \end{align*}
 where $L:QS^m_\C(\R^n)\to\C$ is a linear map and $(\sigma_{[m-j]})_{j\in \N}$ is any positively quasihomogeneous resolution of $\sigma$ (see \cite{PR}).
\end{defn}

\begin{rk}
  A leading symbol form on $ CS_{\C}^m(\R^n)$ induces one on $CS_{\C}^m(\Z^n)$ as follows: Given a linear map  $L:QS^m_\C(\R^n)\to\C$, the linear form on $CS^m_\C(\Z^n)$ defined by 
 \[
  L\circ\, e  \circ (\cdot)_{[m]} = L\circ (\cdot)_{[m]} \circ e
 \]
 (see Proposition \ref{prop:classical}) is a singular linear form on $CS^m_\C(\Z^n)$ independent of the choice of the extension map $e$.
\end{rk}

\subsection{From traces to translation invariant linear forms}
\label{subsectransinv}

In this paragraph we relate $\Z^n$--translation invariant linear forms on symbols with linear forms that vanish on star--brackets. In the following, $m$ is an arbitrary complex number.

Given a linear form $\lambda$ on $CS_{\A_\th}^m(\Z^n)$ and $l\in \Z^n$ we set $T_l^*\lambda(\sigma) =\lambda(T_l\sigma)$, where as before, $T_l$ denotes the translation on symbols $\sigma\mapsto \sigma(\cdot+l)$. Since $CS_{\A_\th}^m(\Z^n)$ is stable under $T_l$ (see Remark \ref{rk:transl}), $T_l^* \la$ is a linear form on $CS_{\A_\th}^m(\Z^n)$.

\begin{defn}
\label{def:trans-closed}
A linear form $\lambda: CS_{\A_\th}^m(\Z^n)\to \C$ is \emph{closed} if $\lambda\circ \bar \delta_j=0$ for all $j\in\{1,\cdots, n\}$, where $\bar \delta_j$ is the map defined in \eqref{bardelta}.

A linear form $\lambda: CS_{\A_\th}^m(\Z^n)\to \C$ is \emph{$\Z^n$--translation invariant} if it satisfies one of the following two equivalent conditions:
\begin{enumerate}
\item $T_l^*\lambda=\lambda$  for all $l\in \Z^n$,
\item $\lambda\circ \Delta_j=0$ for all $j\in \{1,\cdots, n\}$, where $\Delta_j$ is the forward difference operator introduced in \eqref{bigdeltaj}.
\end{enumerate}
\end{defn}

\begin{rk} 
The implication $1. \Rightarrow 2.$ follows from setting $l:=e_j$. The implication $2. \Rightarrow 1.$ follows from setting $l=\sum_{j=1}^n l_je_j$ and using induction on $\vert l\vert=\sum_{j=1}^n l_j$. 
\end{rk} 

Part $(iii)$ of the following lemma shows that traces on noncommutative toroidal symbols are $\Z^n$--translation invariant and closed. Part $(i)$, which is inspired from an analogue statement on the two dimensional noncommutative torus proved in \cite{FW}, yields a noncommutative counterpart for a factorisation through the fibre for traces on pseudodifferential operators on closed manifolds.

\begin{lem}
\label{lem:translationinv}
$(i)$ Let $\lambda$ be a closed linear form on $CS_{\A_\th}^m(\Z^n)$. Then, $\lambda$ factorises in a unique way through $\bar\t$. In other words, there is a unique linear form 
on $\iota_\th^{-1}(CS_{\A_\th}^m(\Z^n))= CS_\C^{m}(\Z^n)$
 \[
 \bar\lambda:=\lambda\circ \iota_\theta:  CS_\C^{m}(\Z^n) \to \C, \quad \text{such that}\quad \lambda=\bar\lambda\circ \bar\t\, .
 \]
$(ii)$ Let $\lambda$ be a closed linear form on $CS_{\A_\th}^m(\Z^n)$. Then for any $\sigma\in CS_{\A_\th}^m(\Z^n)$ and $k\in \Z^n$,
\[
\lambda(\left(T_{k} -I\right) (\sg)) = \lambda\left(\set {\sg\circ_\th U_{-k},U_k}_\th\right)  -\lambda\left( \{\{\sigma, U_k\}_\theta,U_{-k}\}_\th \right).
\]  
$(iii)$ Let $\lambda$ be a $CS_{\A_\th}(\Z^n)$--trace on $CS_{\A_\th}^m(\Z^n)$, i.e. 
\[
 \lambda \left(\{\sigma, \tau\}_\theta\right)=0\quad \text{ for all }\sigma,\tau \in CS_{\A_\th}(\Z^n)\text{ such that } \{\sigma, \tau\}_\theta\in CS_{\A_\th}^m(\Z^n)\, .
\]
Then $\lambda$ is closed and $\Z^n$--translation invariant.
\end{lem}

\begin{proof} 
$(i)$ For any $\sg\in CS_{\A_\th}^m(\Z^n)$, $k\in \Z^n$,
\[
\sg(k) = \sg_0(k) U_0 + \sum_{l\in \Z^n\backslash \set{0}} \dfrac{\sg_l(k)}{p_l} \delta_{I_l}(U_l)
\]
where for all $l\neq 0$, $I_l:=\set{j\in \set{1,\cdots,n} : l_j\neq 0}\neq \varnothing$, $\delta_{I_l}:=\prod_{j\in I_l} \delta_j$, and $p_l:=\prod_{j\in I_l} l_j \neq 0$.
We deduce from this the following equality
\begin{equation}
\label{symbdecomp}
\sg(k) - \sg_0(k) U_0 = \sum_{j=1}^n \delta_j\left(\sum_{l\in A_j} \dfrac{\sg_l(k)}{p_l} \delta_{I_{l}\backslash\set{j}}(U_l)\right)
\end{equation}
where $A_j:=\set{l\in \Z^n\backslash\set{0} : 1,\cdots ,j-1 \notin I_l \text{ and } j\in I_l }$. 
Define for all $j\in \set{1,\cdots,n}$,
\[
\tau^{(j)}(k):= \sum_{l\in A_j} \dfrac{\sg_l(k)}{p_l} \delta_{I_{l}\backslash\set{j}}(U_l) =: \sum_{l\in A_j} \tau^{(j,l)}(k)\, .
\]
We claim that $\tau^{(j)}$ is a classical symbol. 
Note that for all $(j,l)$ such that $j\in \set{1,\cdots,n}$, $l\in A_j$, and all $\a\in \N^n$, $\norm{(p_l)^{-1} \delta^\a\delta_{I_l\backslash \set{j}}(U_l)}\leq \langle l\rangle^{|\a|}$. It follows from Lemma \ref{lem:fouriercoeff} that for all $\a,\b \in \N^n$, and $N\in \N$, there is a constant $C$ such that
\[
p_{\a,\b}^{(\Real(m))}(\tau^{(j,l)}) \leq  C \langle l\rangle^{-N}\, .
\] 
In particular, the family $(\tau^{(j,l)})_{l\in A_j}$ is absolutely summable in $S^{\Real(m)}_{\A_\th}(\Z^n)$, and its sum $\tau^{(j)}\in S^{\Real(m)}_{\A_\th}(\Z^n)$.
Let $e$ be an $(\A_\th,\ol \t)$--compatible extension map. By continuity of $e$, it follows that 
\[
e(\tau^{(j)} )= \sum_{l\in A_j} e(\tau^{(j,l)})\, .
\]
Since $e$ is $(\A_\th,\ol \t)$--compatible, $e(\tau^{(j,l)})= e(\sg)_l (p_l)^{-1} \delta_{I_l\backslash \set{j}}(U_l)$. 
Moreover, since $\sg$ is a classical symbol, $e(\sg)\in CS^m_{\A_\th}(\R^n)$. Let $(\sg_{[m-i]})_i$ be a sequence of symbols such that $\sg_{[m-i]}\in HS^{m-i}_{\A_\th}(\R^n)$, 
and $e(\sg)\sim \sum_i \sg_{[m-i]}$. Fix $q\in \N$ and $l\in \Z^n$. 
We have $e(\sg)_l = \sum_{i=0}^q (\sg_{[m-i]})_l + R^{(q+1)}_l$, where $R^{(q+1)}\in S^{m-q-1}_{\A_\th}(\R^n)$. 
Thus, setting $\rho^{(i,j,l)}:=(\sg_{[m-i]})_l(p_l)^{-1} \delta_{I_l\backslash \set{j}}(U_l)$, we obtain
\[
e(\tau^{(j)}) = \sum_{i=0}^q \sum_{l\in A_j} \rho^{(i,j,l)}  +  \sum_{l\in A_j} R^{(q+1)}_l(p_l)^{-1} \delta_{I_l\backslash \set{j}}(U_l)\, .
\]
Using again Lemma $\ref{lem:fouriercoeff}$, we obtain the absolute summability of $(\rho^{(i,j,l)})_l$ in $S^{\Real(m)-i}_{\A_\th}(\R^n)$, and of $(R^{(q+1)}_l(p_l)^{-1} \delta_{I_l\backslash \set{j}}(U_l))_{l}$ in $S^{\Real(m)-q-1}_{\A_\th}(\R^n)$. 
By Lemma \ref{closedness}, $ \sum_{l\in A_j} \rho^{(i,j,l)}$ belongs to $HS^{m-i}_{\A_\th}(\R^n)$, and the claim follows. 
  
To conclude the proof, note that \eqref{symbdecomp} can now be reformulated as
\[
\sg - \iota_\th \circ \ol \t (\sg) = \sum_{j=1}^n \ol \delta_j (\tau^{(j)})\, .
\]
Applying now $\la$ on either side of this equality yields the result.

\medskip

$(ii)$ A computation shows that for any $k$ in $\Z^n$ and any $\sigma$ in $CS_{\A_\th}^m(\Z^n)$
\[
\sg \circ_\th U_k = (T_{k}\sg)\,  U_k\,\, \qquad \text{ and }\qquad U_k\circ_\th \sg = \sg U_k + \sum_{l\in \Z^n} \sg_l [U_k,U_l],
\]
from which it follows that
\begin{equation}
\label{1}
\{\sg,U_k\}_\th =  \left(T_k-I\right)(\sg )\, U_k + \sum_{l\in \Z^n} \sg_l \, [U_l,U_k].
\end{equation}

Applying $\la$ on either side of $\eqref{1}$ yields
\[
\la(\{\sg,U_k\}_\th) =  \la(\left(T_k-I\right)(\sg )\, U_k )+ \la\left(\sum_{l\in \Z^n} \sg_l \, [U_l,U_k]\right).
\]
By $(i)$, $\la$ factorizes through $\bar \t$, and by \eqref{eq:UkUl},
\[
\bar \t \left(\sum_{l\in \Z^n} \sg_l \, [U_l,U_k]\right) = \bar \t \left(\sum_{l\in \Z^n} \sg_{l-k} \, 2i \sin (\pi l\th k) U_{l}\right) = 0\, ,
\]
therefore we obtain for any $\tau\in CS_{\A_\th}^m(\Z^n)$ and $k\in \Z^n$,
\begin{equation}
\label{firststep}
\la(\left(T_k-I\right)(\tau )\, U_k )= \la(\{\tau,U_k\}_\th).
\end{equation}

If we multiply \eqref{1} by $U_{-k}$ and use the fact that $[U_k,U_{-k}]=0$, we get
\begin{equation}\label{2}
  \left(T_k-I\right)(\sg) = \set{\sg,U_k}_\th \, U_{-k} + \sum_{l \in \Z^n}\sg_l [U_k,U_l U_{-k}]\,.
\end{equation}
Moreover, a direct computation shows that
\[
\tau U_{-k} = \tau \circ_\th U_{-k} - \left((T_{-k}-I) \tau\right) U_{-k}\, 
\]
holds for any symbol in $CS_{\A_\th}^m(\Z^n)$. Applied to the symbol $\tau = \set{\sg, U_k}_\th$ this formula combined with \eqref{2} yields
\begin{align}
  \left(T_k-I\right)\sg&=
   \set{\sg,U_k}_\th \circ_\th  U_{-k} - \left(T_{-k}-I\right)(\set{\sg,U_k}_\th) \, U_{-k} + \sum_{l\in \Z^n} \sg_l [U_k,U_l U_{-k}]\nonumber\\
 &= \set {\sg\circ_\th U_{-k},U_k}_\th -\left(T_{-k}-I\right)(\set{\sg,U_k}_\th) U_{-k} + \sum_{l\in \Z^n} \sg_l [U_k,U_l U_{-k}] \label{Tmoinsk}
\end{align}
since $\set{U_{-k},U_k}_\th=0$.
 
Note that
\begin{equation} \label{barteq}
 \bar \t \left(\sum_{l\in \Z^n} \sg_l [U_k,U_l U_{-k}]\right) = \bar \t \left(\sum_{l\in \Z^n} \sg_{l} \,(e^{-2\pi i k\theta l}-1) U_{l}\right) = 0\, .
\end{equation}
Applying $\la$ on either side of $\eqref{Tmoinsk}$, and using \eqref{barteq} as well as the fact that $\la$ factorizes through $\bar \t$, we get
\begin{equation}
\label{secondstep}
\la(\left(T_k-I\right)(\sg)) = \la(\set {\sg\circ_\th U_{-k},U_k}_\th) -\la(\left(T_{-k}-I\right)( \set{\sg,U_k}_\th ) U_{-k})
\end{equation}

If we replace $k$ by $-k$ in \eqref{firststep} and then apply it to $\tau:=\{\sigma, U_k\}_\theta$, we obtain 
\[
\lambda\left(\left(T_{-k}-I\right)(\{\sigma, U_k\}_\theta)\,U_{-k}\right) =  \lambda\left( \{\{\sigma, U_k\}_\theta,U_{-k}\}_\th \right),
\] 
which, combined with \eqref{secondstep}, yields the desired equality.

\medskip

$(iii)$ By $(ii)$, it is enough to check that $\la$ is closed. But this follows directly from \eqref{bracketkj}.
\end{proof} 

The definition of $\Z^n$--translation invariant, exotic and singular linear forms on $CS_{\A_\th}^m(\Z^n)$ is naturally extended to linear forms defined on $CS_\C^m(\Z^n)$.

\medskip

Combining $(i)$ and $(iii)$ of the previous lemma yields:

\begin{cor}
\label{cor:step2}
Any $CS_{\A_\th}(\Z^n)$--trace $\la$ on $CS_{\A_\th}^m(\Z^n)$ is $\Z^n$--translation invariant and closed. Moreover the linear form  $\ol{\la}:=\la\circ\iota_\th$ is a $\Z^n$--translation invariant linear form on $CS_\C^m(\Z^n)$ satisfying $\la=\ol \la\circ \ol \t$.
If $\la$ is singular (resp.\ exotic, resp.\ $\ell^1$--continuous), then so is $\ol \la$.
\end{cor}

\subsection{Classification of translation invariant linear forms on (commutative) toroidal symbols}
\label{subsecclassif}

As the Corollary \ref{cor:step2} shows, finding $CS_{\A_\th}(\Z^n)$--traces on $CS_{\A_\th}^m(\Z^n)$ reduces to its commutative counterpart, namely finding $\Z^n$--translation invariant linear forms on $CS_\C^m(\Z^n)$, an issue we deal with in this section.

In the following $[\sg]$ denotes the equivalence class of a symbol modulo smoothing symbols. 

By definition, a leading symbol form is singular and we shall need a few properties of singular linear forms. 
 
\begin{lem}
\label{lem:singext}
(i) The map 
\[
\Phi: S_{\C}(\Z^n)/S^{-\infty}_\C(\Z^n) \to S_\C(\R^n)/S^{-\infty}_{\C}(\R^n)
\]
defined by $\Phi([\sg]) = [e(\sg)]$, where $e$ is an extension map, is well defined, and independent of the choice of $e$.

\medskip

(ii) The map $\Phi$ is a linear isomorphism. Moreover, $\Phi^{-1}([\sg]) = [\sg_{|\Z^n}]$ for all $\sg\in CS_\C(\R^n)$.

\medskip

(iii) If $\la$ is a linear form on $CS^m_\C(\Z^n)$, then the linear form on $CS^m_\C(\R^n)$  
\[
\wt \la : \sg \mapsto   \la(\sg_{|\Z^n}) 
\]
satisfies $\la = \wt \la \circ e$ for every extension map $e$. If $\la$ is singular (resp.\ exotic, $\Z^n$--translation invariant), then so is $\wt \la$.
Moreover, when $\la$ is singular, $\wt \la=\dot{\la}\circ \Phi^{-1} \circ [\cdot]$, where $\dot{\la}$ denotes the quotient linear map on $CS_\C^m(\Z^n)/ S^{-\infty}_\C(\Z^n)$ associated to $\la$.
\end{lem}

\begin{proof}
$(i)$ Let $\sg$, $\sg'$ be two symbols such that $\sg\sim \sg'$. Then $e(\sg)\sim e(\sg')$, which implies that $\Phi$ is well defined. Let $e$ and $e'$ be two extension maps, and $\sg\in S_\C(\Z^n)$. Then $e(\sg) \sim e'(\sg)$ by Lemma \ref{lem:extensionmap} $(iii)$, which implies that $\Phi$ is indeed independent of $e$.

\medskip

$(ii)$ The map $\Phi$ is clearly linear. Suppose that $\Phi([\sg])=0$ for a symbol $\sg\in S_\C(\Z^n)$. Then, for an extension map $e$, $e(\sg)$ is smoothing. Thus, by Lemma \ref{restr}, $\sg$ is smoothing, and $[\sg]=0$. Therefore, $\Phi$ is injective. 
It remains to show that $\Phi$ is surjective. Let $[\sg]\in S_\C(\R^n)/S^{-\infty}_{\C}(\R^n)$, and let $e$ be an extension map. We have $\sg\sim e(\sg_{|\Z^n})$ by Lemma \ref{lem:extensionmap} $(iii)$, and therefore $[\sg] = [e(\sg_{|\Z^n})] =\Phi([\sg_{|\Z^n}])$.

\medskip

$(iii)$ This is straightforward.
\end{proof}
 
The singularity of a linear form is preserved under composition with the extension map.

\begin{lem}
\label{lem:singular and tranlation inv}
 Let $\mu:CS^m_\C(\R^n)\to\C$ be a linear form and let $e$ be any extension map. If $\mu$ is singular (resp.\ exotic, $\Z^n$--translation invariant), then so is $ \mu \circ e$.
\end{lem}

\begin{proof}
The facts that $\mu\circ e$ is singular, {resp.\  exotic}, resp.\ $\Z^n$--translation invariant follow respectively from the facts that $e$ preserves the order (Proposition \ref{prop:classical}), and that $e$ commutes with translations (Definition \ref{def:extension}).
\end{proof}

The noncommutative residue (resp.\ the canonical integral) introduced in (\ref{eq:ncr}) (resp.\ (\ref{eq:cutoff})) is an $\R^n$-- (and hence $\Z^n$--) translation invariant linear form on symbols on $CS_\C^m(\R^n)$ with $m\in \Z$ (resp.\ $m\in \C\setminus \Z_n$) in view of \cite[Corollary 2.59]{P2} (resp.\ \cite[Theorem 2.61]{P2}). For exotic linear forms there is a one to one correspondence between $\Z^n$--translation invariance, $\R^n$--translation invariance and Stokes' property of a linear form \cite[Proposition 5.34]{P2}, a property we are about to define.

\begin{defn} 
A linear form $\la$ on $CS_\C^m(\R^n)$ is said to satisfy \emph{Stokes' property} if for all $i=1,\ldots,n$ and for all $\sg\in CS_\C^m(\R^n)$, $\la\circ\partial_{\xi_i}(\sg)=0$.
\end{defn} 

For the sake of completeness we provide the proof of the fact that $\Z^n$--translation invariance implies Stokes' property for exotic linear forms since we will use this fact explicitly.

\begin{prop}\cite[Proposition\ 5.34]{P2}
\label{prop:exotic transl inv Stokes}
Let $\lambda$ be an exotic linear form on $CS_{\C}^m(\R^n)$. If $\lambda$ is $\Z^n$--translation invariant then it satisfies Stokes' property. 
\end{prop}

\begin{proof}
Let $\sigma$ be a symbol in $CS_{\C}^m(\R^n)$. For any $p$ in $\Z^n$ the translated symbol $T_{p}\sigma:= \sigma(\cdot+p)$ also lies in $CS_{\C}^m(\R^n)$ as can be seen from a Taylor expansion of $\xi\mapsto \sigma(\xi+p)$ at $p=0$ \cite[Proposition 5.52]{P2}.
More precisely, there is an integer $K\geq 2$ such that for any $p \in \Z^n$ the remainder term $R^{p}_K(\sigma):= T_{p}\sigma-\sum_{\vert \alpha\vert=0}^{K-1} \partial_\xi^\alpha\sigma\frac{p^\alpha}{\alpha!}$ lies in $L^1(\R^n)$. Since $\lambda$ is exotic and $\Z^n$--translation invariant this implies that for any $p$ in $\Z^n$
we have
\[
0=\lambda\left(  T_{p}\sigma-\sigma\right) =\sum_{\vert \alpha\vert=1}^{K-1}
\lambda( \partial_\xi^\alpha\sigma)\, \frac{p^\alpha}{\alpha!}+
\lambda(R^{p}_K(\sigma))=\sum_{\vert \alpha\vert=1}^{K-1}
\lambda( \partial_\xi^\alpha\sigma)\, \frac{p^\alpha}{\alpha!}. 
\]  
Thus $\lambda(\partial_\xi^\alpha\sigma)=0$ for any $\alpha\in \N^n$ such that $1\leq \vert \alpha\vert <K$. Choosing $\alpha=(0,\cdots, 0, 1,0\cdots,0)$ with the $1$ at the $i$--th slot yields $\lambda(\partial_{\xi_i}\sigma)=0$ and hence Stokes' property.
\end{proof}
 
The following proposition characterises $\Z^n$--translation invariant linear forms on symbols on $\R^n$. Recall that $\Z_n=\Z\cap[-n,+\infty[$. For any $m\in\C$ we denote by $H_\C^{m}(\R^n)$ the space of smooth functions $f$ from $\R^n\backslash \set{0}$ into $\C$, such that $f(t\xi)= t^mf(\xi)$ for all $t>0$ and $\xi\in \R^n\backslash \set{0}$. 

\begin{prop}  
\label{prop:charresRn} 
(i) Let $m\in \Z_n$. Any $\Z^n$--translation invariant exotic linear form on $CS_\C^m(\R^n)$ is a linear combination of a leading symbol form and the restriction of the noncommutative residue to $CS_\C^m(\R^n)$.\\ 
In particular, any $\Z^n$--translation invariant exotic linear form on $CS_\C^\Z(\R^n)$ is proportional to the noncommutative residue (compare with \cite[Proposition 5.40]{P2}).

\medskip

(ii) Let $m\in \C\setminus \Z_n$. Any $\Z^n$--translation invariant exotic linear form on $CS_\C^m(\R^n)$ is a leading symbol form.

 \medskip

(iii) Any $\Z^n$--translation invariant exotic linear form on $CS_\C^{\notin \Z_n}(\R^n)$ vanishes.
\end{prop}

\begin{proof} 
We note that $(iii)$ easily follows from $(ii)$ since leading symbol forms do not extend beyond a symbol set of fixed order. Let us prove $(i)$ and $(ii)$. 

Let $m\in \C$ and let $\lambda$ be a $\Z^n$--translation invariant exotic linear form on $CS_\C^m(\R^n)$.
Since $\lambda$ is exotic, it induces a linear form $\lambda_j$ on every ${H}_\C^{m-j}(\Z^n)$  with $j\in \N$ defined by $\lambda_j(f)= \lambda(f\, \chi )$ for any $f$ in ${H}_\C^{m-j}(\R^n)$ and any excision function $\chi$ around $0$. Moreover, the induced linear form $\lambda_j $ vanishes on every ${H}_\C^{m-j}(\R^n)$ with $\Real(m)+n<j$. Thus, applied to a symbol $\sigma\sim \sum_j f_{[m-j]} \chi$ in $CS_\C^m(\R^n)$ the linear form $\lambda$ vanishes if $\Real(m) < -n$ and reads 
 
\begin{equation}
 \label{eq:lambdaj}\lambda(\sigma)= \sum_{0\leq j\leq \Real(m)+n}\lambda_j (f_{[m-j]})\quad \text{if} \quad \Real(m) \geq -n.
\end{equation}
 
We henceforth assume that $\Real(m) \geq -n$. By Proposition \ref{prop:exotic transl inv Stokes} the linear form $\lambda$ satisfies Stokes' property as a result of its $\Z^n$--translation invariance. Since $\partial_{\xi_i} \chi$ has compact support it follows that 
 \[
 \lambda_j(\partial_{\xi_i} g )=\lambda(\partial_{\xi_i} g\, \chi )= -\lambda( g \,\partial_{\xi_i} \chi )=0 
 \]
for any $i\in \{1,\cdots, n\}$ and any $g\in {H}_\C^{m-j+1}(\R^n)$ with $j\in \N^*$. Let $f\in {H}_\C^{m-j}(\R^n)$ for some $j\in \N$. If $m-j\neq -n$, then $f(\xi)=\frac{1}{m-j+n}\sum_{i=1}^n\partial_{\xi_i}(\xi_i f(\xi))$ so that $\lambda_j(f)=0$. If $m-j=-n$ then $h(\xi):= f(\xi)-{\res}(f)\,\vert \xi\vert^{-n}$ has vanishing residue. It follows from \cite[Lemma 1.3]{FGLS} that $h=\sum_{i=1 }^n\partial_{\xi_i} h_i$ for some homogeneous functions $h_i\in {H}_\C^{-n+1}(\R^n)$. Thus $\lambda_j(h)= 0$ and hence  $\lambda_j(f)= C\, {\res}(f)$ with $C:= \lambda_j(\xi\mapsto \vert \xi\vert^{-n})=\lambda(\xi\mapsto\vert \xi\vert^{-n}\, \chi(\xi))$ independent of $\chi$ and $j$. Setting $f= f_{[m-j]}$ for some $j\in \N$, $\sg_{[m]}:=f_{[m]} \chi$, and inserting the result back in (\ref{eq:lambdaj}) yields 
\[
 \lambda(\sigma-\sigma_{[m]})=  \sum_{0<j\leq \Real(m)+n}\lambda_j (f_{[m-j]})=0 \quad \text{if} \quad m\notin \Z_n,
\] 
and
\[
 \lambda(\sigma-\sigma_{[m]})=  C\,  \sum_{0<j\leq \Real(m)+n}{\res}(f_{[m-j]})\,\delta_{m-j+n,0}= C\,{\res}(\sigma) \quad \text{if} \quad m\in \Z_n,
\]
where $\delta_{m-j+n,0}$ is the Kronecker delta function.

Summing up we find that $\lambda(\sigma)=\lambda(\sigma_{[m]}) $ if $m\notin \Z_n$  and $\lambda(\sigma)=\lambda(\sigma_{[m]})+ C{\res}(\sigma) $ otherwise, which yields the announced characterisation.
\end{proof}

\begin{rk}
 In the case that $m=-n$, the restriction of the noncommutative residue to $CS_\C^{-n}(\R^n)$ is an example of a leading symbol form, so we have that any $\Z^n$--translation invariant exotic linear form on $CS_\C^{-n}(\R^n)$ is a leading symbol form.
\end{rk}

The following theorem  yields a characterisation of  $\Z^n$--translation invariant exotic linear forms on integer order toroidal symbols, which is new to our knowledge.
\begin{thm}
\label{thm:charres}
The \emph{toroidal noncommutative residue}, defined for all symbols $\sigma\in CS_\C(\Z^n)$ by
 \begin{equation}
 \label{res}
 \restor(\sigma):=\res\circ e\,(\sg)=\int_{\Sbb^{n-1}} e(\sigma)_{[-n]}(\xi)\, dS(\xi),  \index{$\restor$}
 \end{equation}
where $e$ is an extension map, is independent of the choice of $e$. It is a $\Z^n$--translation invariant exotic linear form (see Definition \ref{exotic}) on $CS_\C(\Z^n)$.

\medskip 
(i)  Let $m\in \Z_n$. Any $\Z^n$--translation invariant exotic linear form on $CS_\C^{m}(\Z^n)$ is a linear combination of a leading symbol form and the restriction of the toroidal residue $\restor$ to $CS_\C^{m}(\Z^n)$. 

\medskip
(ii) Let $m\in \C\setminus \Z_n$. Any $\Z^n$--translation invariant exotic linear form on $CS_\C^{m}(\Z^n)$ is a leading symbol form.

\medskip
(iii) Any $\Z^n$--translation invariant exotic linear form on $CS_\C^{\Z_n}(\Z^n)$ is proportional to the restriction of the toroidal residue $\restor$ to $CS_\C^{\Z_n}(\Z^n)$.
\end{thm}

\begin{proof}
The facts that $\restor$ is exotic and $\Z^n$--translation invariant follow from the $\Z^n$--translation invariance of the residue $\res$ (\ref{eq:ncr}) on symbols on $\R^n$ combined with Lemma \ref{lem:singular and tranlation inv}.
The fact that $\restor$ is independent of the choice of the extension map follows from Lemma \ref{lem:extensionmap} $(iii)$.
\medskip

Let $m\in \C$. By Lemma \ref{lem:singext} $(iii)$, given an exotic $\Z^n$--translation invariant linear form $\la$ on $CS_\C^m(\Z^n)$, the corresponding linear form $\wt \la$ on $CS_\C^m(\R^n)$ is exotic and $\Z^n$--translation invariant. The statements $(i)$ and $(ii)$ then follow from Proposition \ref{prop:charresRn}  which classifies $\wt \la$ according to whether $m$ lies in $\Z_n$ or not. If $\wt \la$ is proportional to $\res$ then $\la$ is proportional to $\restor$.
If $\wt \la$ is a linear combination of the leading symbol form and $\res$ then so is $\la$ a linear combination of the induced leading symbol form on toroidal symbols and $\restor$. The statement $(iii)$ is a consequence of the fact that a leading symbol form on $CS_\C^m(\Z^n)$ does not extend to $CS_\C^{\Z_n}(\Z^n)$. 
\end{proof}
 
\begin{lem}
\label{lem:extendedlinearforms} 
  Let $m$ be a complex number with real part smaller than $-n$. Any $\Z^n$--translation invariant $\ell^1$--continuous linear form on the space of symbols $CS_\C^m(\Z^n)$ is proportional to ${\textstyle\sum_{\Z^n}}$ (the standard summation over $\Z^n$).  
\end{lem}

\begin{proof} 
Given a symbol $\sigma\in CS_\C^m(\Z^n)$ with $\Real(m)<-n$, we define for all $N\in \N$, $\sigma_N := \sum_{k\in [-N, N]^n\cap \Z^n}\sigma(k)\delta_{k}$, where $\delta_{k}$ is the function: $\delta_k:p\mapsto\delta_{k,p}$, which maps a point $p$ to the Kronecker delta function $\delta_{k,p}$. By linearity and translation invariance, $\lambda(\sigma_N)=\la(\delta_0)\sum_{k\in [-N, N]^n\cap \Z^n}\sigma(k)$. 
By $\ell^1$--continuity, taking the limit as $N$ goes to infinity yields the result.
\end{proof}

For a normalised extension map $e$, we have $\sum_{\Z^n}\sigma=\int_{\R^n} e(\sigma)$ for any $\sigma$ in $CS^{<-n}_\C(\Z^n)$ (see Definition \ref{extensionmaps}), which motivates the following definition. Recall from Section \ref{subsec:linear forms on symbols} that ${\textstyle\cutoffsum_{\Z^n}}$ is the discrete cut--off sum which defines a $\Z^n$--translation invariant linear extension on $CS^{\notin \Z_n}_\C(\R^n)$ and which extends the ordinary discrete sum $\sum_{\Z^n}$.

The following theorem yields a characterisation of $\Z^n$--translation invariant $\ell^1$--continuous linear forms  on non--integer order toroidal symbols, which is new to our knowledge.

\begin{thm}
\label{thm:characressum}
(i) The \emph{toroidal canonical discrete sum}, defined on $CS^{\notin \Z_n}_\C(\Z^n)$  as 
\begin{equation}
\label{discretesum}
 \cutoffsumtor:=\cutoffint_{\R^n} \circ\, e \index{$\cutoffsumtor$}
\end{equation}
where $e$ is a normalised extension map, is independent of the choice of $e$. 
It is a $\Z^n$--translation invariant $\ell^1$--continuous linear form (Definitions \ref{ell 1-continuous} and \ref{def:trans-closed}) on $CS^{\notin \Z_n}_\C(\Z^n)$ and coincides with the usual summation on symbols whose order has real part $<-n$.
\medskip

(ii) Let $m\in \C\setminus \Z_n$. Any $\Z^n$--translation invariant $\ell^1$--continuous linear form on the space $CS^m_\C(\Z^n)$ is proportional to the toroidal canonical discrete sum $\cutoffsumtor$. 

Consequently, any $\Z^n$--translation invariant $\ell^1$--continuous linear form on $CS^{\notin \Z}_\C(\Z^n)$ is proportional to the toroidal canonical discrete sum $\cutoffsumtor$. In particular we have 
\begin{equation}
\label{eq:torsumdissum}
 \cutoffsumtor={\textstyle\cutoffsum_{\Z^n}} \circ\, e
\end{equation}
where $e$ is any (non necessarily normalised) extension map.
\end{thm}

\begin{proof}
$(i)$ The fact that $ \cutoffsumtor$ coincides with the summation over $\Z^n$ on symbols whose order has real part $<-n$ follows from the fact that $e$ is a normalised extension map (Definition \ref{extensionmaps}). In particular, $\cutoffsumtor$ is $\ell^1$--continuous (Definition \ref{ell 1-continuous}). Moreover, $\cutoffsumtor$ is $\Z^n$--translation invariant as a consequence of the $\Z^n$--translation invariance of the cut--off integral (\ref{eq:cutoff}) on non--integer order symbols on $\R^n$ combined with Lemma \ref{lem:singular and tranlation inv}. 
To check that it does not depend on the choice of $e$, let $e$, $e'$ be two normalised extension maps and define
\[
 \la_0:= \cutoffint_{\R^n} \circ\, e  - \cutoffint_{\R^n}  \circ\, e' .
\] 
It is clear from what precedes that the restriction of $\la_0$ to classical symbols of order in $\Z\cap ]-\infty,-n[$ is zero. Moreover, $\la:=(\la_0)_{|CS^{\notin \Z}_\C(\Z^n)}$ 
is an exotic $\Z^n$--translation invariant linear form on $CS^{\notin \Z}_\C(\Z^n)$. By Proposition \ref{prop:charresRn} $(iii)$, $\la$ vanishes and hence
\[
 \cutoffint_{\R^n}\circ\, e  = \cutoffint_{\R^n} \circ\, e' .
\] 

$(ii)$ Let $m\in \C\setminus \Z_n$. Any $\Z^n$--translation invariant $\ell^ 1$--continuous linear form $\lambda$ on $CS^m_\C(\Z^n)$ restricts to a $\Z^n$--translation invariant $\ell^1$--continuous linear form on the space $ CS^{m-[m]-n-1}_\C(\Z^n)=CS_\C^m(\Z^n)\cap CS^{<-n}_\C(\Z^n)$. By Lemma \ref{lem:extendedlinearforms} this restriction is proportional to the standard summation $\sum_{\Z^n}$, hence the existence of a constant $C$ such that $\lambda_{\vert_{CS^{<-n}_\C(\Z^n)} }= C\sum_{\Z^n}$. The linear form $\lambda_1:= \lambda- C \, {\cutoffsumtor }_{|CS^m_\C(\Z^n)}$ is therefore an exotic $\Z^n$--translation invariant linear form on $CS^m_\C(\Z^n)$. As a consequence, its extension $\wt \la_1$ is an exotic $\Z^n$--translation invariant linear form on $CS^m_\C(\R^n)$.
By Proposition \ref{prop:charresRn} $(iii)$, $\wt \la_1$ is a leading symbol form, which implies that $\la_1$ is a leading symbol form. Therefore, $\lambda$ is a linear combination of a leading symbol form and the toroidal canonical discrete sum. Since $\lambda$ is $\ell^1$--continuous, it is actually proportional to the toroidal canonical discrete sum $\cutoffsumtor$.

Consequently, any $\Z^n$--translation invariant $\ell^1$--continuous linear form on $CS^{\notin \Z}_\C(\Z^n)$ is proportional to $\cutoffsumtor$. 

Let $e$ be a (non necessarily normalised) extension map $e$. From what we just proved and the $\Z^n$--translation invariance of the $\ell^ 1$--continuous linear form ${\textstyle\cutoffsum_{\Z^n}}\circ e$ on $CS^{\notin \Z_n}_\C(\Z^n)$ (which comes from the $\Z^n$--translation invariance of the linear form ${\textstyle\cutoffsum_{\Z^n}}$ on $CS^{\notin \Z_n}_\C(\R^n)$ \cite[Corollary 5.35]{P2}), it follows that there exists $\la\in \R$ such that $\cutoffsum \circ e= \la\ \cutoffsumtor$. Since for any symbol $\sg$ whose order has real part $<-n$, $\cutoffsum \circ e (\sg) = \sum_{k\in \Z^n} \sg(k) = \cutoffsumtor \sg$, it follows that $\la=1$, which proves 
Equation (\ref{eq:torsumdissum}).
\end{proof}

\begin{rk}
We could equally well have taken (\ref{eq:torsumdissum}) as a definition setting
\[ 
\cutoffsumtor={\textstyle\cutoffsum_{\Z^n}  } \circ\, e
\]
for any (non necessarily normalised) extension map $e$, and derived (\ref{discretesum}) as a property. 
\end{rk}

\section{Classification of traces on (noncommutative) toroidal symbols and operators}
\label{section: classification nct}

For fixed $m\in\C$, we consider traces on $CS^m_{\A_\th}(\Z^n)$ and $\Cl_{\th}^m(\T^n)$ in the sense of resp.\ $CS_{\A_\th}(\Z^n)$-- and $\Cl_{\th}(\T^n)$--traces (see Definition \ref{def:trace}).

\subsection{Main classification result}

The linear forms on symbols in $CS_{\C}(\Z^n)$ introduced in Section \ref{subsec:linear forms on symbols} induce traces on corresponding subsets of $CS_{\A_\th}(\Z^n)$ and of $\Cl_{\th}(\T^n)$. In this section we describe these traces and their classification.

\begin{rk}
\label{rk:reduction to symbols}
 Since $\Op_\th$ is a topological and algebraic isomorphism between $CS^m_{\A_\th}(\Z^n)$ and $\Cl^m_{\th}(\T^n)$ for all $m\in\C$ (see Proposition \ref{propquantisation} and Theorem \ref{thmcomp}), we can reduce the problem of the classification of traces on subsets of $\Cl_{\th}(\T^n)$, to the problem of the classification of traces on subsets of $CS_{\A_\th}(\Z^n)$.
\end{rk}

\begin{prop}
\label{prop:noncommutative residue trace}
 The linear form on $CS_{\A_\th}(\Z^n)$ defined by 
\[
 \res_\th (\sg):= \int_{\Sbb^{n-1}} \ol \t(e(\sg))_{[-n]} (\xi) \, dS(\xi) =\res \circ\, \ol \t \circ\, e (\sg) \index{$\res_\th$}
\]
is independent of the chosen $\ol \t$--compatible extension map $e:S_{\A_\th}(\Z^n)\to S_{\A_\th}(\R^n)$. It is an exotic (and hence singular) trace on $(CS_{\A_\th}(\Z^n),\circ_\th)$, called the \emph{symbolic noncommutative residue}. 

The linear form on $\Cl_\th(\T^n)$ defined by 
\[
 \Res_\th := \res_\th\circ \Op_\th^{-1} \index{$\Res_\th$}
\]
is an exotic trace on $\Cl_{\th}(\T^n)$, called the \emph{noncommutative residue}. 
\end{prop}

\begin{proof}
The second statement follows from the first one using the bijectivity of the map ${\rm Op}_\theta$ (see Remark \ref{rk:reduction to symbols}). Let us prove the first statement.  
Thanks to Lemma \ref{lem:permute extensions and traces}, Proposition \ref{prop:classical} and Lemma \ref{lem:permut} $(ii)$ we can permute the three operations $\ol \t$, $e$, and $(\cdot)_{[-n]}$ in $\res_\th$. This way, for a $\ol \t$--compatible extension map $e$, for any $\sg\in CS_{\A_\th}(\Z^n)$ we have
\begin{equation}
 \label{permutres}
 \res_{\th}(\sg) = \int_{\Sbb^{n-1}} \ol \t (e(\sg)_{[-n]}) \, dS
\end{equation}
where $e(\sg)_{[-n]} \in [e(\sg)]_{-n}$ is chosen in $HS^{-n}_{\A_\th}(\Z^n)$.

Since from Lemma \ref{lem:permute extensions and traces} we have $\res_\th=\res\, \circ\ e_\C\circ\, \ol \t=\restor\circ \, \ol \t$, the fact that it is an exotic linear form independent of the extension map $e$ chosen to define $\restor$, is a direct consequence of Corollary \ref{cor:step2} and Theorem \ref{thm:charres}. 
 
To prove that $\res_\th$ is a trace we use Stokes' property on the unit sphere as in the usual proof of the cyclicity of the noncommutative residue $\res$.
 
Let $\sg\in CS^r_{\A_\th}(\Z^n)$ and $\sg^\prime\in CS^{r^\prime}_{\A_\th}(\Z^n)$. By \eqref{eqasympt} and \eqref{permutres}, we have for a given extension map $e$,
\begin{align}
\label{eq:residuetrace}
\res_\th \set{\sg ,\sg^\prime}_\th = \sum_{|\a|+i+i'=r+r'+n} \frac{1}{\a!}\int_{\Sbb^{n-1}}  \ol \t\Big(&(\del_\xi^\a e(\sg)_{[r-i]}) \, \ol\delta^\a e(\sg^\prime)_{[r'-i']} \nonumber \\
&\ -(\del_\xi^\a e(\sg^\prime)_{[r'-i']})\, \ol\delta^\a e(\sg)_{[r -i]}\Big)\, dS\, ,
\end{align}
the sum being set to zero when $r+r'+n\notin \N$. Using Stokes' property on the unit sphere, namely the fact that $\int_{\S^{n-1}} \del_{\xi_j} f\, dS=0$, 
when $f$ is positively homogeneous of degree $-n+1$, we see that the result follows from several integrations by parts with respect to the variable $\xi$ in \eqref{eq:residuetrace}, and applications of the formulae $\ol \t (\ol \delta_j(\rho) \rho')=-\ol \t(\rho \ol\delta_j(\rho'))$ and 
$\ol \t(\rho \rho')=\ol \t(\rho'\rho)$, which are valid for all symbols $\rho,\rho'$ in $S_{\A_\th}(\R^n)$.
\end{proof}

\begin{prop}
\label{prop:canonical trace}
 The linear form on $CS^{\notin\Z_n}_{\A_\th}(\Z^n)$ defined by
\[
 {\textstyle\cutoffsum_\th }:=  \cutoffint_{\R^n} \circ\ \ol \t \circ e \index{${\textstyle\cutoffsum_\th }$}
\]
is independent of the choice of the normalised $\ol \t$--compatible extension map $e$. It is an $\ell^1$--continuous trace, called the \emph{canonical discrete sum}. This trace coincides with $\sum_{\Z^n} \circ\ \ol \t$ on symbols whose order has real part $<-n$. 

The linear form on $\Cl_\th^{\notin\Z_n}(\T^n)$ defined by 
\[
\TR_\th := {\textstyle\cutoffsum_\th}\circ \Op_\th^{-1} \index{$\TR_\th$}
\]
is an $\mathcal{L}^1$--continuous trace on $\Cl_{\th}^{\notin\Z_n}(\T^n)$, called the \emph{canonical trace}. This trace coincides with the operator trace on operators whose order has real part $<-n$.
\end{prop}

\begin{proof} 
The second statement follows from the first one using the bijectivity of the map ${\Op}_\theta$ (see Remark \ref{rk:reduction to symbols}). Let us prove the first statement. 
Thanks to Lemma \ref{lem:permute extensions and traces}, for a $\ol \t$--compatible extension map $e$ we have
\begin{equation}
 \label{permutcutoff}
 {\textstyle\cutoffsum_\th }= \cutoffint_{\R^n} \circ\, e_\C \circ\ \ol \t .
\end{equation} 
By definition if $e$ is normalised, $e_\C$ is also normalised, and therefore $\cutoffsum_\th=\cutoffsumtor\circ \, \ol \t$.
By Theorem \ref{thm:characressum} $(i)$ and Corollary \ref{cor:step2}, the linear form $\cutoffsum_\th=\cutoffsumtor\circ \, \ol \t$ is independent of the normalised extension map $e$ chosen to define $\cutoffsumtor$, and it is an $\ell^1$--continuous linear form. Indeed, on symbols whose order has real part $<-n$, it coincides with $\int_{\R^n} \circ\, \ol \t \circ\ e$. Similarly, since  ${\TR}_\theta={\textstyle\cutoffsum_\th}\circ \Op_\th^{-1}={\textstyle\cutoffsumtor} \circ\ \ol \t\circ \Op_\th^{-1}$, Theorem \ref{thm:Sobolev} implies that $\TR_\th$ and the ordinary trace coincide on operators whose order has real part $<-n$.

We now prove that $ {\textstyle\cutoffsum_\th}$ is a trace. Let $\sg,\sg^\prime$ in $CS_{\A_\th}(\Z^n)$ be such that $\set{\sg , \sg^\prime}_\th\in CS_{\A_\th}^{\notin\Z_n}(\Z^n)$, then by Theorem \ref{thm:asympt}
\begin{equation}
 {\textstyle\cutoffsum_\th} \set{\sg, \sg^\prime}_\th \sim  
 \sum_{\a\in \N^n} \frac{1}{\a !}\cutoffint_{\R^n} \ol \t\Big((\del_\xi^\a e(\sg)) \, \bar\delta^{\a} e(\sg^\prime)-(\del_\xi^\a e(\tau)) \, \bar\delta^{\a} e(\sg)\Big).
\end{equation}
As in the commutative case, this term vanishes since the finite part of the integral over a ball of radius sufficiently large of homogeneous terms of non--integer degree vanishes. 
\end{proof}

\begin{prop}
\label{leading symbol trace}
Given a linear map $L:QS^m_\C(\R^n)\to\C$, the linear form on $CS^m_{\A_\th}(\Z^n)$ defined by 
\[
 L\circ\, e\circ \, \ol \t\, ((\cdot)_{[m]})
\]
is a singular trace on $CS^m_{\A_\th}(\Z^n)$ called a \emph{leading symbol trace} on $CS^m_{\A_\th}(\Z^n)$.

The linear form on $\Cl_{\th}^m(\T^n)$ defined by 
\[
 L\circ\, e\circ \, \ol \t\, (\Op_\th^{-1}(\cdot)_{[m]})
\]
is a singular trace on $\Cl_{\th}^m(\T^n)$ called a \emph{leading symbol trace} on $\Cl_{\th}^m(\T^n)$.
\end{prop}

\begin{proof} 
By Theorem \ref{thm:asympt} (see also Remark \ref{rk:order star bracket}), for any $\sg\in CS^m_{\A_\th}(\Z^n)$ and $\tau\in CS_{\A_\th}^{m'}(\Z^n)$, the commutator $\set{\sg,\tau}_\th$ belongs to the space $CS^{m+m'}_{\A_\th}(\Z^n)$, and moreover by Remark \ref{rk:bar t a trace},
\[
 \ol \t \left(\set{\sg,\tau}_\th\right) \in CS^{m+m'-1}_{\C}(\Z^n)
\]
so the homogeneous term of degree $m$ in any positively quasihomogeneous resolution of $\ol \t \left(\set{\sg,\tau}_\th\right)$ vanishes. A leading symbol form $L\circ\, e\circ \, \ol \t ((\cdot)_{[m]})$ therefore vanishes on star--brackets of symbols. It follows from \eqref{bracketOp} that the leading symbol form $ L\circ\, e\circ \, \ol \t (\Op_\th^{-1}(\cdot)_{[m]})$ vanishes on operator brackets. This justifies calling these linear forms leading symbol traces on their respective algebras.
\end{proof}

\begin{rk}
\label{rk:leading symbol singular not exotic}
Leading symbol traces on $\Cl_{\th}^m(\T^n)$ with $\Real(m)<-n$ provide examples of traces that are singular but not exotic since they do not vanish on the set $\Cl_{\th}^m(\T^n)$ whose elements are trace--class operators.
By Lemma \ref{lem:extendedlinearforms}, leading symbol traces are not $\ell^1$-- (resp.\ $\mathcal{L}^1$)--continuous for symbols (resp.\ for operators), hence the fact that under an $\ell^1$-- (resp.\ $\mathcal{L}^1$)--continuity assumption, these traces are ruled out of the classification result below.
\end{rk}

We now state our main result which is the noncommutative toroidal generalisation of known characterisations of traces on classical pseudodifferential operators on $\T^n$ (for general closed manifolds see \cite{LN-J,N-J} for the case of fixed order, and \cite{P2, Sc} for an overview).
 
\begin{thm}
\label{mainthm1} Let $m$ be a complex number.
\begin{enumerate}
\item If $m\in \Z_n$, any exotic trace on $CS_{\A_\th}^m(\Z^n)$ is a linear combination of a leading symbol trace and the restriction of $\res_\th$ to $CS_{\A_\th}^m(\Z^n)$.
 
Consequently, any exotic trace on $CS_{\A_\th}^{\Z_n}(\Z^n)$ is proportional to the restriction of $\res_\th$ to $CS_{\A_\th}^{\Z_n}(\Z^n)$.  

\item If $m\in \Z_n$, any exotic trace on  $\Cl_\th^m(\T^n)$ is a linear combination of a leading symbol trace and the restriction of $\Res_\th$ to $Cl_{\th}^m(\T^n)$. 

Consequently, any exotic trace on $\Cl_\th^{\Z_n}(\T^n)$ is proportional to the restriction of $\Res_\th$ to $\Cl_\th^{\Z_n}(\T^n)$. 

\item If $m\notin \Z_n$, any $\ell^1$--continuous trace on $CS^{m}_{\A_\th}(\Z^n)$ is proportional to the restriction of\ \ ${\textstyle\cutoffsum_\th }$ to $CS^{m}_{\A_\th}(\Z^n)$. 

Consequently, any $\ell^1$--continuous trace on $CS^{\notin\Z}_{\A_\th}(\Z^n)$ is proportional to the restriction of\ \  ${\textstyle\cutoffsum_\th }$ to $CS^{\notin\Z}_{\A_\th}(\Z^n)$.

\item If $m\notin \Z_n$, any $\mathcal{L}^1$--continuous trace on $\Cl_\th^{m}(\T^n)$ is proportional to the restriction of $\TR_\th$ to $\Cl_\th^{m}(\T^n)$.  
  
Consequently, any $\mathcal{L}^1$--continuous trace on $\Cl_\th^{\notin\Z}(\T^n)$ is proportional to the restriction of $\TR_\th$ to $\Cl_\th^{\notin \Z}(\T^n)$.
\end{enumerate}
\end{thm} 
  
\begin{proof}
As noted in Remark \ref{rk:reduction to symbols}, items $2.$ and $4.$ are respectively straightforward consequences of items $1.$ and $3.$. We now prove $1.$ and $3.$.

Let $\la$ be a $CS_{\A_\th}(\Z^n)$--trace on $CS_{\A_\th}^m(\Z^n)$. By Corollary \ref{cor:step2}, $\la$ is closed and $\Z^n$--translation invariant and the linear form $\ol{\la}:=\la\circ\iota_\th$ defines a $\Z^n$--translation invariant linear form on $CS_\C^m(\Z^n)$ satisfying $\la=\ol \la\circ \ol \t$. If $\la$ is exotic (resp.\ $\ell^1$--continuous), so is $\ol \la$ exotic (resp.\ $\ell^1$--continuous).

Let $m\in \Z_n$ and let $\la$ be exotic. By Theorem \ref{thm:charres} $(i)$, the linear form $\ol{\la}$ is a linear combination of a leading symbol form and $\restor$, so that $\la$ is a linear combination of a leading symbol form and $\res_\th$.

Let $m\in \C\setminus \Z_n$ and let $\la$ be $\ell^1$--continuous. By Theorem \ref{thm:characressum} $(ii)$, the linear form $\ol{\la}$ is proportional to $\cutoffsumtor$, so that $\la$ is proportional to $\cutoffsum_\theta$.

The tracial properties of these linear forms follow from Proposition \ref{prop:noncommutative residue trace}, Proposition \ref{prop:canonical trace} and Proposition \ref{leading symbol trace}.
\end{proof}

\begin{rk}
 Item 2.\ in Theorem \ref{mainthm1} compares with the classification result by Fathizadeh and Wong \cite[Theorem 4.4]{FW} since both give a characterisation of traces on $\Cl_\th^{\Z}(\T^2)$ ($n=2$). Our classification holds under the assumption that the trace be exotic whereas Fathizadeh and Wong's holds under the assumption that the trace be singular and continuous. \\ Further fixing the order of the operators as in Theorem \ref{mainthm1} offers a generalisation of these classification results on traces on noncommutative tori.
\end{rk}

\subsection{The commutative case}

Theorem \ref{mainthm1}  has a commutative counterpart obtained by setting $\theta=0$ which is interesting for its own sake since it yields a characterisation of traces on toroidal symbols of fixed order. It also  yields back known characterisations of the noncommutative residue \cite{W1,W2} and the canonical trace \cite{KV,LN-J,MSS,N-J} on certain classes of pseudodifferential operators on the torus seen as a particular closed manifold. Our results in the commutative case are nevertheless weaker than those quoted here since we require that the trace be either exotic or $\ell^1$--continuous.

As before  we shall identify $\A_0$ and $\A=\Ci(\T^n)$ (see Remark \ref{rk:A=A0}).
Recall that $CS_{\A}^m(\R^n)$ can be identified with the usual space of classical symbols of order $m$ and $\Cl_{\A}^m(\T^n)$ with the usual space of classical pseudodifferential operators of order $m$ on the commutative torus $\T^n$. 

A symbol $\sigma$ in  $CS_{\A}^m(\Z^n)$ is an $\A$--valued map on $\Z^n$ and any extension $e (\sg)$ an $\A$--valued map on $\R^n$. Taking the trace $\t$ amounts to integrating over the torus $\T^n$ so that in view of Lemma \ref{lem:permut} which allows us to permute the integration, the map $\sigma\mapsto \sigma_{[-n]}$ and the extension map $e$, for $\theta=0$ the symbolic noncommutative residue reads for $\sigma \in CS_{\A }(\Z^n)$
 \[
 \res_0 (\sg)= \int_{\Sbb^{n-1}}\int_{\T^n} (e (\sg))_{[-n]} (\xi) \, dS(\xi) =\res \left(\int_{\T^n}\, e  (\sg)\right) 
 \]
and the toroidal canonical discrete sum reads for $\sigma \in CS^{\notin \Z_n}_{\A }(\Z^n)$,
\[
{\textstyle\cutoffsum_0 }(\sg)= \cutoffint_{\R^n} \, \int_{\T^n}e (\sigma)={\textstyle\cutoffsum_{\Z^n}}\int_{\T^n} e(\sigma), 
\]
independently of the choice of normalised extension map $e$.

Similarly, a leading symbol linear form on $CS^m_{\A}(\Z^n)$ is of the form 
 \[
 \sigma\longmapsto  L\left(\int_{\T^n}(e (\sigma))_{[m]}\right), 
 \]
for some linear form $L:QS^m_\C(\R^n)\to\C$.

In the following, $\Res$ denotes the usual noncommutative Wodzicki residue \cite{W1,W2}, and $\TR$  the usual Kontsevich and Vishik canonical trace \cite{KV,LN-J,MSS,N-J} on closed manifolds (here the torus).

\begin{cor}
\label{maincor1} Let $m$ be a complex number.
\begin{enumerate}
\item If $m\in \Z_n$, any exotic trace on $CS_\A^m(\Z^n)$ is a linear combination of a leading symbol linear trace and the restriction of $\res_0$ to $CS_\A^m(\Z^n)$.  
 
Consequently, any exotic trace on $CS_\A^{\Z_n}(\Z^n)$ is proportional to the restriction of $\res_0$ to $CS_\A^{\Z_n}(\Z^n)$. 

\item For an operator $A$ in $\Cl(\T^n)$, we have 
  \begin{equation}
  \label{eq:identif1}
  \Res_0(A)= \res_0 \left(\Op_0^{-1}(A)\right)= \Res(A).
  \end{equation} 

If $m\in \Z_n$, any exotic trace on  $\Cl^m(\T^n)$ is a linear combination of a leading symbol linear trace and the restriction of $\Res$ to $\Cl^m(\T^n)$. 
 
Consequently, any exotic trace on $\Cl^{\Z_n}(\T^n)$ is proportional to the restriction of $\Res$ to $\Cl^{\Z_n}(\T^n)$.  

\item If $m\notin \Z_n$, any $\ell^1$--continuous trace on $CS^{m}_\A(\Z^n)$ is proportional to the restriction of\ \ $\cutoffsum_0$ to $CS^{m}_\A(\Z^n)$. 

Consequently, any $\ell^1$--continuous trace on $CS^{\notin\Z}_\A(\Z^n)$ is proportional to the restriction of\ \ $\cutoffsum_0$ to $CS^{\notin\Z}_\A(\Z^n)$. 

\item For an operator $A$ in $\Cl^{\notin\Z_n} (\T^n)$, we have
  \begin{equation}
  \label{eq:identif2}
  \TR_0(A)= {\textstyle\cutoffsum }_0 \Op_0^{-1}(A)=\TR(A).
  \end{equation}

If $m\notin \Z_n$, any $\mathcal{L}^1$--continuous trace on $\Cl^{m}(\T^n)$ is proportional to the restriction of $\TR$ to $\Cl^{m}(\T^n)$. 
  
Consequently, any $\mathcal{L}^1$--continuous trace on $\Cl^{\notin\Z}(\T^n)$ is proportional to the restriction of $\TR$ to $\Cl^{\notin \Z}(\T^n)$. 
\end{enumerate} 
\end{cor}

\begin{proof} 
This is a straightforward consequence of Theorem \ref{mainthm1} in which we have set $\theta=0$.

The identifications \eqref{eq:identif1} (resp.\ \eqref{eq:identif2}), follow from the fact that $\Res_0$ and $\Res$ (resp.\ $\TR_0$ and $\TR$), enjoy the same characterisation as traces on operators on the torus.
\end{proof}

\section{Traces of holomorphic families}
\label{section:holomorphic}

Let as before $\B$ stand for $\A_\theta$ or $\C$.

\subsection{Holomorphic families} 

The notion of holomorphic family of classical pseudodifferential symbols was first introduced by Guillemin in \cite{Gu} and extensively used by Kontsevich and Vishik in \cite{KV}. It uses the notion of holomorphic family of symbols which we extend here to discrete symbols. We need the concept of local uniformity of families of symbols.

\begin{defn}
A family $\sigma(u)$ of symbols in $S_\B(\R^n)$, resp.\ $S_\B(\Z^n)$ of real order $m(u)$ parametrized by $u$ in an open subset $U$ of some topological space is called \emph{locally uniform in a neighborhood of $u_0\in U$} if there is a compact subset $K$ of $U$ containing $u_0$ such that
\begin{eqnarray}
\label{eq:uniformestimate}
 &&\forall \alpha\in \Z_{\geq 0}^n,\forall \beta\in \Z_{\geq 0}^n, \exists\, C_{K, \alpha, \beta }>0,\nonumber\\
 && \forall u\in K, \quad p^{ m(u)}_{\a,\b}( \sg(u)) \leq C_{K,\alpha, \beta}.
\end{eqnarray}
The family is called \emph{locally uniform on $U$} if this holds for any $u_0\in U$.
\end{defn}

\begin{rk}
For a family $\sigma(u)$ of constant order $m(u)=m$, the condition (\ref{eq:uniformestimate}) implies that the extension $e:S_\B(\Z^n)\to S_\B(\R^n)$ is continuous at $u_0$ in the Fr\'echet topology of symbols of constant order.
\end{rk}

\begin{defn}
We call an extension map $e:S_\B(\Z^n) \to S_\B(\R^n)$ \emph{locally uniform} if it takes any locally uniform family $\sigma(u)$ in $S_\B(\Z^n)$ on an open subset $U$ of any topological space to a locally uniform family on $U$ of symbols $e\circ \sigma(u)$ in $S_\B(\R^n)$. 
\end{defn} 

\begin{lem}
\label{lem:elocallyuniform} 
The normalised extension map $e$ introduced in (\ref{eq:enormalised}) is locally uniform.
\end{lem}

\begin{proof} 
The proof of the uniform estimate (\ref{eq:uniformestimate}) closely follows that of the continuity of $e$ in Lemma \ref{lem:extensionmap}.\\
Given $\a,\b\in \N^n$, we first observe that for any real number $m$ and any $\sigma \in S^m_{\A_\th}(\Z^n)$ 
\begin{align*}
\del^\b (\ol\delta^\a e(\sg)) (\xi) &= \sum_{k\in \Z^n} (\del^\b \wh \rho) (\xi-k)\, (\ol \delta^\a\sg)(k) \\
&= \sum_{k\in \Z^n} (\ol \Delta^\b \phi_\b)(\xi-k)\, (\ol \delta^\a\sg)(k)\\
&= (-1)^{|\b|}\sum_{k\in \Z^n} \phi_\b(\xi-k)\, (\Delta^\b \ol \delta^\a\sg)(k)
\end{align*}
where $\ol \Delta_j = I- T_{-e_j}$, $\ol \Delta^\b = \ol \Delta_1^{\b_1}\cdots \ol \Delta_n^{\b_n}$, and where the $\phi_\b$ are functions in $\SS(\R^n)$. 
Using the notation $\xi-\Z^n := \set{\xi-k \ :\ k\in \Z^n}$ and Peetre's inequality, this implies that for all $\sg\in S^m_{\A_\th}(\Z^n)$, and all $\a,\b\in \N^n$,
\begin{align*}
 \norm{\del^\b (\ol\delta^\a e(\sg))(\xi)}&\leq \sum_{k\in \Z^n} | \phi_\b(\xi-k)|\, \norm{(\Delta^\b \ol \delta^\a(\sg))(k)}\\
 &\leq p^{(m)}_{\a,\b}(\sg) \sum_{k\in \Z^n} | \phi_\b(\xi-k)|\,\langle k\rangle^{m-|\b|}\\
 &\leq p^{(m)}_{\a,\b}(\sg) \sum_{\eta\in \xi-\Z^n} | \phi_\b(\eta)|\,\langle \xi -\eta\rangle^{m-|\b|}\\
 &\leq \langle \xi\rangle^{m-|\b|} p^{(m)}_{\a,\b}(\sg) 2^{|m-|\b||} \sum_{\eta\in \xi-\Z^n} | \phi_\b(\eta)|\, \langle \eta\rangle^{|m-|\b||} \\
 &\leq \langle \xi\rangle^{m-|\b|} \sup_{\xi\in \R^n} g_{\b,m}(\xi)\, p^{(m)}_{\a,\b}(\sg) 
\end{align*}
where 
\[
 g_{\b,m}:\xi\mapsto 2^{|m-|\b||} \sum_{\eta\in \xi-\Z^n} |\phi_\b(\eta)|\, \langle \eta\rangle^{|m-|\b||}.
\]
Implementing this estimate for $\sigma(u)$ of order $m(u)$ with $u$ in a compact neighborhood $K$ of $u_0\in U$, we get the locally uniform estimate (\ref{eq:uniformestimate}) where we have set:
\[
 C_{ K, \alpha,\beta}= \sup_{\xi\in \R^n, u\in K}  g_{\b,m(u)}(\xi).
\] 
\end{proof} 

We now extend the notion of holomorphicity to families of $\B$--valued classical symbols.
\begin{defn}
\label{def:holomorphicfamily} 
A family of classical symbols  
\begin{equation}
\label{eq:holclassical}
 \sigma(u)\sim \sum_{j\geq 0}
 \sigma_{[m(u)-j]}(u) \ \in CS_\B^{m(u)}(\R^n),
\end{equation}
of complex order $m(u)$ parametrised by an open subset $U$ of some topological space is \emph{continuous}, resp.\ \emph{differentiable} at a point $u_0\in U$ if:
\begin{enumerate}
\item its order $m:U\to \C$ is continuous, resp.\ differentiable at $u_0$;
\item for every $\xi\in \R^n$ the maps $u\mapsto\sigma(u)(\xi)$ and $u\mapsto \sigma_{[m(u)-j]}(u)(\xi)$ for $j\in \Z_{\geq 0}$ are continuous, resp.\ differentiable at $u_0$ as functions with values in the Fr\'echet space $\B$;    
\item for any integer $N\geq 1$, the remainder $\sigma_{(N)} (u):= \sigma(u)- \sum_{j=0}^{N-1}\chi\, \sigma_{[m(u)-j]}(u) $, resp.\ its derivative $\partial_u\sigma_{(N)} (u):= \partial_u\sigma(u)- \sum_{j=0}^{N-1} \partial_u\sigma_{[m(u)-j]}(u)$ defines a locally uniform family of order $\Real(m(u))-N+\e$ in a neighborhood of $u_0$, for any positive $\e$.
\end{enumerate}
If $U$ is a complex domain, a differentiable family $\sigma(u)$ at $u_0$ is said to be holomorphic at $u_0$.\\
If $\sigma(u)$ is continuous, resp.\ differentiable, resp.\ holomorphic at every point $u\in U$, it is called a continuous, resp.\ differentiable, resp.\ holomorphic family of symbols parametrised by $U$. \\
Similarly, a family of discrete symbols  
\begin{equation}
\label{eq:holclassicaldiscrete}
 \sigma(u)\sim \sum_{j\geq 0}
 \sigma_{[m(u)-j]}(u) \ \in CS_\B^{m(u)}(\Z^n),
\end{equation}
of complex order $m(u)$ parametrised by an open subset $U$ of some topological space is continuous, resp.\ differentiable at a point $u_0\in U$ if it satisfies conditions 1.--3.\ where in the second condition $\xi\in \R^n$ is replaced by $k\in \Z^n$.
\end{defn} 

Condition 3.\ implies the following: Let $\sigma^h_{[m(u)-j]}(u)$ denote the homogeneous function which coincides with the quasihomogeneous function $\sigma_{[m(u)-j]}(u)$ outside some ball centered at zero. For any excision function $\chi$ at zero \footnote{By excision function at zero we mean a smooth function on $\R^n$ which vanishes in a neighborhood of zero and is identically one outside some ball centered at zero.} for any integer $N\geq 1$, the remainder, resp.\ its derivative
\begin{eqnarray*}
\sigma^\chi_{(N)} (u)&:=& \sigma(u)- \sum_{j=0}^{N-1}\chi\, \sigma^h_{[m(u)-j]}(u),\\
{\rm resp.} \quad \partial_u\sigma^\chi_{(N)} (u)&:= & \partial_u\sigma(u)- \sum_{j=0}^{N-1}\chi\, \partial_u\sigma^h_{[m(u)-j]}(u) 
\end{eqnarray*} 
should be locally uniform of order $\Real(m(u))-N+\e$ in a neighborhood of $u_0$, for any positive $\e$. If this requirement is fulfilled for one given excision function $\chi$,  it then follows from the local uniformity of the family, that it holds for any other excision function $\chi^\prime$ since the differences $(\chi-\chi^\prime)\, \sigma^h_{[m(u)-j]}(u)$ and $(\chi-\chi^\prime)\,\partial_u \sigma^h_{[m(u)-j]}(u)$ are smoothing symbols. 

\begin{rk} 
For a family $\sigma(u)$ of continuous, resp.\ differentiable symbols  
\begin{equation}\label{eq:logclassical}
 \sigma(u)\sim \sum_{j\geq 0}
 \sigma_{[m -j]}(u) \ \in CS_\B^{m }(\R^n),
\end{equation}
of {\rm constant} complex order $m$, the continuity, resp.\ differentiability at a point $u_0\in U$ amounts to its continuity, resp.\ differentiability in the Fr\'echet topology of classical symbols of constant order $m$.
\end{rk}

\begin{rk} 
Given a differentiable family $\sigma(u)$ in $CS_{\A_\theta}(\R^n)$ (resp.\ $CS_{\A_\theta}(\Z^n)$), the family $\bar \t (\sigma(u))$ is differentiable in $CS_{\C}(\R^n)$ (resp.\ $CS_{\C}(\Z^n)$) and we have 
\begin{equation}
\label{eq:derivativetrace}
\partial_u(\bar \t(\sigma(u)))= \bar \t(\partial_u(\sigma(u))).
\end{equation}
\end{rk}

\begin{prop}
\label{prop:holdiscreteversuscont}
A locally uniform extension map $e:CS_\B (\Z^n)\to  CS_\B (\R^n)$ takes a continuous, resp.\ differentiable, resp.\ holomorphic family of symbols in $CS_\B (\Z^n)$ to one in $CS_\B (\R^n)$. If the family is differentiable we have 
\begin{eqnarray}
\label{eq:extensionpartial}
 \partial_u (e\circ \sigma(u))&\sim& e\circ\partial_u(\sigma(u));\\
 \partial_u (e\circ \sigma(u))_{[m(u)-j]} &\sim& e\circ\left(\partial_u\sigma(u) \right)_{[m(u)-j]}\sim e\circ \partial_u (\sigma_{[m(u)-j]}(u)) \nonumber 
\end{eqnarray} 
for any $j$ in $\Z_{\geq 0}$. In particular, this holds for the normalised extension map $e$ introduced in (\ref{eq:enormalised}). 
\end{prop}

\begin{proof}    
Let $\sigma(u)$ be a continuous, resp.\ differentiable, resp.\ holomorphic family in $CS_\B (\Z^n)$, so that it satisfies conditions 1.--3.\ in Definition \ref{def:holomorphicfamily}. Since the extension does not modify the order, the family $e\circ \sigma(u)$ also satisfies Condition 1. Likewise, the continuity of the extension map in the Fr\'echet topology of symbols of constant order ensures that $e\circ \sigma(u)$ satisfies Condition 2. The fact that $e\circ \sigma(u)$  satisfies Condition 3.\  then follows from the local uniformity assumption on $e$. In particular, this assumption holds for the normalised extension map as a consequence of  Lemma \ref{lem:elocallyuniform}. \\
The first identity in (\ref{eq:extensionpartial}) follows from the fact that $\partial_u (e\circ \sigma(u))$ and $e\circ\partial_u (\sigma(u))$ define two extensions of $\partial_u (\sigma(u))$ and hence differ by a smoothing symbol. The second identity in (\ref{eq:extensionpartial}) then follows from the fact that the extension map preserves the order combined with the fact that for a differentiable family $\tau(u)$ in $CS_\B (\R^n)$ we have $\left(\partial_u\tau(u) \right)_{[m(u)-j]}\sim \partial_u (\tau_{[m(u)-j]}(u))$ for any $j\in \Z_{\geq 0}$.
\end{proof}

\begin{cor}
A family $\sigma(u)$ in $CS_\B (\Z^n)$ is continuous, resp.\ differentiable, resp.\ holomorphic at a point if and only if there is a locally uniform extension map $e$ such that $e\circ \sigma(u)$ in $CS_\B (\R^n)$ is continuous, resp.\ differentiable, resp.\ holomorphic at that point.
\end{cor}

\begin{proof}
In view of Proposition \ref{prop:holdiscreteversuscont} which shows one implication, the equivalence follows from the simple observation that when restricted to $\Z^n$, a family $\sigma(u)$ in $CS_\B (\R^n)$ which is continuous, resp.\ differentiable, resp.\ holomorphic at a point, gives rise to a family $\sigma(u)_{\vert_{\Z^n}}$ in $CS_\B (\Z^n)$ which is continuous, resp.\ differentiable, resp.\ holomorphic at that point.
\end{proof}

\begin{defn}  
We call a family $A(u)$ of operators in $\Cl_{\theta}(\T^n)$ of order $m(u)$ parametrized by an open subset $U$ of a topological space continuous, resp.\ differentiable at $u_0\in U$ (resp.\ on $U$) if the family of symbols $\Op_\theta^{-1}(A(u))$ is continuous, resp.\ differentiable at $u_0 $ (resp.\ on $U$) in $CS_{\A_\theta}(\Z^n)$ of order $m(u)$.\\ If $U\subset \C$ we call a differentiable family at $u_0$ (resp.\ on $U$) holomorphic at $u_0$ (resp.\ on $U$).
\end{defn}

\begin{rk}
For a differentiable family $A(u)$ of operators in $\Cl_\theta(\T^n)$ we have
\begin{equation}\label{eq:diffOp}
 \partial_u\left(\Op_\theta^{-1}(A(u))\right)= \Op_\theta^{-1}\left(\partial_uA(u)\right).  
\end{equation}
\end{rk}

\subsection{The noncommutative residue as a complex residue}

Given a pseudodifferential operator $A$ on the noncommutative torus and a holomorphic\footnote{holomorphic in the sense of the previous section.} germ $A(z)$ at zero of pseudodifferential operators on the noncommutative torus, which coincides with $A$ at $z=0$, the subsequent theorem shows that ${\TR}_\theta (A(z))$ defines a meromorphic germ at zero so that we can define  the regularised trace $\underset{z=0}{\fp}\left({\TR}_\theta (A(z))\right)$ of $A$. We use here a terminology and notations similar to those of \cite{PS} to which we refer the reader for further details. 

The subsequent theorem provides  an interpretation of the noncommutative residue  of an operator ${\Res}_\theta(A(0))$ as a complex residue and, whenever $A$ is a differential operator, of the regularized trace of $A$ in terms of the noncommutative residue ${\Res}_\theta(A^\prime(0))$. 

Without any loss of generality, we shall assume that the extension maps $e:S_{\A_\theta}(\Z^n)\to S_{\A_\theta}(\R^n)$ used in the following are normalised and  $\ol \t$--compatible.
In the following, we shall say that a symbol $\sg \in S_{\A_\th}(\Z^n)$ is \emph{polynomial} if is a finite sum 
of symbols of the form $k\mapsto k^\a a$ where $a\in \A_\th$ and $\a\in \N_n$. We refer to \cite[Definition 3.1]{L} for the definition of a log-polyhomogeneous symbol of degree $(m,k)$, where $m$ denotes the polynomial degree, and 
$k$ denotes the logarithmic degree.

\begin{thm}\label{thm:KVPS}
\begin{enumerate}
\item Given a holomorphic germ $\sigma(z)$ in $CS_{\A_\theta}(\Z^n)$ at zero with non-constant affine order $m(z)$, 
\begin{enumerate}
 \item the map 
  \[
   z\longmapsto {\textstyle\cutoffsum_\th} \sigma(z)
  \]
  is meromorphic with a simple pole at zero and residue
 \[
  \underset{z=0}{\Res}\Big({\textstyle\cutoffsum_\th} \sigma(z)\Big)= -\frac{1}{m^\prime(0)} \, {\res}_\theta(\sigma(0)).
 \]
 \item Moreover if $\sigma(0)$ is a polynomial, then  
  \begin{equation}\label{PS}
  \underset{z=0}{\fp}\Big({\textstyle\cutoffsum_\th } \sigma(z)\Big)= -\frac{1}{m^\prime(0)} \, {\res}_\theta(\sigma^\prime(0)),
  \end{equation}  
where we have set ${\res}_\theta(\sigma^\prime(0)):= {\res} \left(e\circ \bar\t(\sigma^\prime(0))\right)$. The residue of the log--polyho\-mogeneous symbol $e\circ \bar\t(\sigma^\prime(0))$ of logarithmic degree one, is defined as in \cite{PS} by integrating on the unit sphere the $(-n)$--th homogeneous component of the symbol by means of  the normalised canonical volume measure on the sphere.
\end{enumerate}
\item Given a holomorphic germ $A(z)$ in $\Cl_\theta(\T^n)$ at zero with non--constant affine order $m(z)$,
\begin{enumerate}
 \item the map 
  \[
   z\longmapsto {\TR}_\theta (A(z))
  \]
  is meromorphic with a simple pole at zero and residue
  \begin{equation}\label{KV}
   \underset{z=0}{\Res}\Big({\TR}_\theta (A(z))\Big)= -\frac{1}{m^\prime(0)} \, {\Res}_\theta(A(0)).
  \end{equation}
 \item Moreover if $A(0)$ is a differential operator, then 
  \[
   {\Res}_\theta(A^\prime(0)):= {\res}_\theta\left(\left(\Op_\theta^{-1}(A)\right)^\prime(0)\right)
  \]
  is well--defined and 
  \begin{equation}\label{PSOp}
    \underset{z=0}{\fp}\Big({\TR}_\theta (A(z))\Big)= -\frac{1}{m^\prime(0)} \, {\Res}_\theta(A^\prime(0)).
  \end{equation}
\end{enumerate}
\end{enumerate}
\end{thm}

\begin{proof} 
\begin{enumerate}
\item The proof relies on known results \cite{PS} (and \cite{P2} for a review) for holomorphic families of symbols on $\R^n$ since the extended symbol $e_\C\circ \ol\t (\sigma(z))=\ol \t \circ e (\sigma(z))$ (see Lemma \ref{lem:permute extensions and traces}) defines a holomorphic family in $CS_\C(\R^n)$.
\begin{enumerate}
\item  By \cite[Theorem 4.10 parts (1) and (2)]{P2}, the map ${\textstyle\cutoffsum_\th}\sigma(z)= \cutoffint \ \ol\t\circ e \circ \sigma(z)$ is meromorphic with a simple pole at zero and 
\begin{eqnarray*}
\underset{z=0}{\Res}\Big({\textstyle\cutoffsum_\th} \sigma(z)\Big)
&=& \underset{z=0}{\Res}\cutoffint \ \ol \t\circ e \circ \sigma(z)\\
&=&-\frac{1}{m^\prime(0)} \, {\res} (\ol \t\circ e \circ\sigma(0))\\
&=&-\frac{1}{m^\prime(0)} \, {\res}_\theta(\sigma(0)).
\end{eqnarray*}
\item  By \cite[Theorem 4.10 part (3)]{P2}, the finite part at $z=0$ is given by  
\begin{eqnarray*}
\underset{z=0}{\fp}\Big({\textstyle\cutoffsum_\th}\sigma(z)\Big)
&=& \underset{z=0}{\fp}\cutoffint \ \ol\t\circ e \circ \sigma(z)\\
&=&-\frac{1}{m^\prime(0)} \, {\res} \left((\ol\t\circ e \circ\sigma)^\prime(0)\right)\\
&=&-\frac{1}{m^\prime(0)} \, {\res} \left(\ol\t\circ e (\sigma^\prime(0))\right) \\
&& {\rm using\quad (\ref{eq:derivativetrace})\quad{\rm and} \quad (\ref{eq:extensionpartial})}\\
&=&-\frac{1}{m^\prime(0)} \, {\res}_\theta(\sigma^\prime(0)).
\end{eqnarray*}
\end{enumerate}
\item 
\begin{enumerate}
\item This follows from 1.\ (a)  applied to $\sigma(z)=\Op_\theta^{-1}(A (z))$ since ${\TR}_\theta (A(z))={\textstyle\cutoffsum_\th }\Op_\theta^{-1}(A(z))$ and ${\Res}_\theta(A(0))={\res}_\theta(\Op_\theta^{-1}(A(0)))$.

\item Similarly, this follows from 1.\ (b) applied to $\sigma(z)=\Op_\theta^{-1}(A (z))$ using (\ref{eq:diffOp}) which yields $\sigma^\prime(z)=\Op_\theta^{-1}(A^\prime (z))$.
\end{enumerate}
\end{enumerate}
\end{proof}
 
Let $P =1+\mathbf{\Delta}$ be as in Example \ref{ex:secprem}. Applying Theorem \ref{thm:KVPS} to the holomorphic family $A(z)= A\, P^{-z }$ with $A$ any operator in $\Cl_\theta (\T^n)$ gives rise to the $\zeta$--function
\[
 \zeta_\theta(A,P)(z)= {\TR}_\theta (A\, P^{-z}). \index{$\zeta_\theta(A,P)(z)$}
\]
It is meromorphic with a simple pole at $z=0$ where the residue is given by 
\begin{equation}
\label{eq:zetaAP}
\underset{z=0}{\Res}\left({\TR}_\theta (A\, P^{-z})\right)= 2\, {\Res}_\theta(A )
\end{equation} 
since $A(z)$ has order $a-2z$, $a$ being the order of $A$. We define the $P$--regularised $\zeta$--trace of an operator $A$ in $\Cl_\theta (\T^n)$ by
\begin{equation}
\label{eq:zetatrace}
{\TR}_\theta^P(A):= \underset{z=0}{\fp}\,\zeta_\theta(A,P)(z). \index{${\TR}_\theta^P$}
\end{equation} 
If the order of $A$ is non--integer or if $A$ is a differential operator then $A$ has vanishing residue, so that by Theorem \ref{thm:KVPS} applied to the holomorphic family $A(z)= A\, P^{-z }$, $\zeta_\theta(A,P)(z)$ is holomorphic at $z=0$ and 
\[
 {\TR}_\theta^P(A)=\zeta_\theta(A,P)(0)= {\TR}_\theta (A )= {\textstyle\cutoffsum_{\Z^n}^{\rm tor}}\, \ol\t(\sg_A)
\]
coincides with the canonical trace of $A$. In particular for $A=I$ the $\zeta$--function $\zeta_P(z):= \zeta_\theta(I,P)(z)$ is holomorphic at zero and we have: 
\[
 \zeta_P(0):={\TR}_\theta^P(I). \index{$\zeta_P$}
\]
If the order of $A$ is smaller than $-n$ then $A$ is trace--class, ${\TR}_\theta (A\, P^{-z})$ has vanishing residue. Combining Theorems \ref{thm:Sobolev}, Proposition \ref{prop:canonical trace} and \ref{thm:KVPS} yields
\[
 {\TR}_\theta^P(A)= \lim_{z\to 0}{\TR}_\theta (A\, P^{-z})= \TR_\theta (A)= \Tr(A)= {\textstyle\sum_{\Z^n}^{\rm tor}}\, \ol\t(\sg_A)
\]
which coincides with the ordinary operator trace of $A$.\\
However, unlike the residue, the linear form ${\Tr}_\theta^P$ does not define a trace on $\Cl_\theta(\T^n)$ and it depends on the regulator $P$. The following corollary shows that the trace defect and the dependence on the regulator can be expressed in terms of a residue.

\begin{cor}
\label{cor:defect}
\begin{enumerate}
\item The Hochschild coboundary of the linear form ${\TR}_\theta^P$ on $\Cl_\theta^a (\T^n)$ is given by a residue. Indeed, if  $A,B\in \Cl_\theta(\T^n)$ the operator bracket $ [B,\log P]$ is classical and we have
\[
 \delta {\TR}_\theta^P(A,B):= {\rm TR}_\theta^P([A,B])= - \frac{1}{2} \, {\Res}_\theta( A\, [B,\log P ]). \index{$ \delta {\TR}_\theta^P(A,B)$}
\]
\item  For any self-adjoint  element  $h $ in $\mathcal{A}_\theta$, we set $\k :=e^{-\frac{h }{2}}$. The operator $P_{\k}:= \k \,P\, \k $ like $P$, is an invertible self--adjoint elliptic operator of order two in $\Cl_\theta (\T^n)$.  Its logarithm $\log P_{\k }$ defined as $\partial_zP_{\k }^z\vert_{z=0}$, coincides with the logarithm of $P_\k$ defined by means of the functional calculus. The difference $\log P -\log P_{\k }$ is classical and for all $A\in \Cl_\theta(\T^n)$ we have 
\[
 {\TR}_\theta^{P_{\k }}(A)- {\TR}_\theta^P(A)=\frac{1}{2} \, {\Res}_\theta\left( A\, \left(\log P -\log P_{\k }\right)\right) .
\]
In particular, $ \zeta_{P_{\k }}(0)= \zeta_{P}(0)$.
\end{enumerate}
\end{cor}
  
\begin{proof} 
\begin{enumerate}
\item The fact that the operator $[B,\log P]$ is classical follows from the fact that the symbol of $\log P $, derived by functional calculus, is $\tau(k):=2\log \langle  k\rangle$.  Indeed, if $\sigma$ denotes the symbol of $B$, it follows from (\ref{starproduct}) that the symbol of $B\, \log P-\log P\, B$ given by 
\begin{eqnarray*} 
(\sg \circ_\theta \tau-\tau\circ_\theta \sigma)(k)  
&=& 2\sum_{l\in \Z^n} \sigma_{l}(k)\,  \log\langle l+k\rangle\,  U_l -2\sum_{l\in \Z^n} \log \langle k\rangle_l\,  \sg(l+k)\,  U_l \\
&=& 2 \sum_{l\in \Z^n} \left( \log\langle l+k\rangle-  \log \langle k\rangle\right) \,  \sg_l(k) \,  U_l \\
\end{eqnarray*}
is classical.  Applying Theorem \ref{thm:KVPS} to the holomorphic family $A(z)= A\, [B,P^{-z}]$ for which $A(0)=0$ yields 
\begin{eqnarray*}
{\TR}_\theta^P([A,B])&=& \underset{z=0}{\fp}\left({\TR}_\theta ([A,B]\, P^{-z})\right)\\
&=& \underset{z=0}{\fp}\left({\TR}_\theta (A\,[B, P^{-z}])\right)\\
&=& -\frac{1}{2} \, {\Res}_\theta( A\, [B,\log P ]).
\end{eqnarray*} 
\item The difference $\log P_{{\k }}-\log P$ is classical. Indeed, as a consequence of Remark \ref{rk:leadingsymbol} the leading symbol of $P_{\k }$ is $e^{-\frac{h}{2}}  \langle  k\rangle^2 e^{-\frac{h}{2}}$. The  symbol of  $\log P_{{\k }}$ differs from  $2\log \langle k\rangle$ by a zero order classical symbol, hence the difference $\log P_{{\k}}-\log P$ is classical. Applying Theorem \ref{thm:KVPS} to the holomorphic family $A(z)= A\, \left(P_{{\k } }^{-z}-P^{-z}\right)$ for which $A(0)=0$ yields 
\begin{eqnarray*} 
{\TR}_\theta^{P_{{\k } }}(A)- {\TR}_\theta^P(A)
&=& \underset{z=0}{\fp}\left({\TR}_\theta \left(A\, \left(P_{{\k } }^{-z}-P^{-z}\right)\right)\right)\\
&=& \frac{1}{2} \, {\Res}_\theta\left( A\, \left(\log P -\log P_{{\k } }\right)  \right).
\end{eqnarray*}
When $A=I$ it follows that
\begin{eqnarray*} 
\zeta_{P_{{\k } }}(0)-\zeta_P(0)&= & \frac{1}{2} \, {\Res}_\theta\left(\left(\log P -\log P_{{\k } }\right)\right)\\
&=& \frac{1}{2} \,\int_0^1 \frac{d}{dt}{\Res}_\theta\left(\left(\log P_{{t{\k } }}\right) \right)\, dt\\
&=& \frac{1}{2} \,\int_0^1 {\Res}_\theta\left(\frac{d}{dt} \left(\log P_{{t{\k } }}\right) \right)\, dt\\
&=& \frac{1}{2} \,\int_0^1 {\Res}_\theta\left(P_{t{\k } }^{-1} \frac{d}{dt}P_{{t{\k } }}\right) \, dt\\
&=& -\frac{1}{4} \int_0^1 {\Res}_\theta\left(P_{{t{\k } }}^{-1} \left(h P_{{t{\k } }}+ P_{{t{\k } }}h  \right)  \right) \, dt\\
&=& -\frac{1}{2}\int_0^1 {\Res}_\theta\left(h \right)\, dt\quad{\rm since}\quad {\Res}_\theta \quad {\rm is} \quad {\rm cyclic}\\
&=& 0,
\end{eqnarray*}
using the fact that $\frac{d}{dt}$ commutes with ${\Res}_\theta$ as a result of the continuity of the normalised extension map and the compactness of the unit sphere on which the $(-n)$--th homogeneous part of the symbol is integrated. 
\end{enumerate}
\end{proof}

By means of Corollary \ref{cor:defect} part  2), we easily recover the conformal invariance of the $\zeta$--function of $1+{\bf \Delta}$ shown in \cite{CT}, \cite{CM1,CM2} and \cite{FK1,FK2} (there it is actually derived for the $\zeta$--function of ${\bf \Delta}$). Given a self--adjoint element $h=h^*$ in $\mathcal{A}_\theta$ we modify the normalised trace $\t$ introduced in \eqref{trace} into $\t_h$ defined for $a$ in $A_\theta$ by
\[
 \t_h(a):= \t(ae^{-h}).
\]
Note that this is not a trace anymore since for any $a,b$ in $A_\theta$
\[
 \t_h(ab)= \t(\k ab\k)=   \t_h(b\sigma_1(a) ),
\]
where we have set $\sigma_t(a):=e^{-th}ae^{th}$ and where as before $\k=e^{-\frac{h}{2}}$. 
The inner product on $A_\th$ is modified correspondingly
\[
 \langle a,b\rangle_h:=\t_h(b^*a)=\t(b^*ae^{-h})= \langle a\k,b\k\rangle \qquad \forall\, a,b\in A_\th,
\]
and we set $\H_{\t_h}$ to be the GNS Hilbert space corresponding to this modified inner product. The right multiplication by $\k$
\[
 R_{\k}(a)= a\, \k, \quad a\in A_\theta
\]
extends to an isometry $W:\H_{\t}\to \H_{\t_h}$ and gives a unitary $\mathcal{A}_\theta$--bimodule isomorphism. \\ 
The operators $\delta_j^h$ are the operators $\delta_j$ seen as operators on $\H_{\t_h}$ so that $ \delta_j^h\circ W=\delta_j\circ R_{\bf \Delta}$, $j=1,\cdots, n$. Thus
\[
 \mathbf{\Delta}_h:= \sum_{j=1}^n \left(\delta_j^h\right)^*\delta_j^h
\]
relates to $\mathbf{\Delta}_0=\mathbf{\Delta}$ by 
\[
 W^*\mathbf{\Delta}_hW=\sum_{j=1}^n  ( \delta_j^h\circ W)^* \delta_j^h\circ W= R_{\k}\mathbf{\Delta} R_k.
\]
Let $J$ be the involution on $\H_\t$ given by $J(a)=a^*$ for any $a$ in $A_\theta$ then $JR_{\k}J(a)=\left( a^*{\k}\right)^*= {\k}\, a$ for all $a\in A_\theta$ so that $JR_{\k}J=\k$ and 
\[
 J R_{\k} \mathbf{\Delta} R_{\k} J= J R_{\k} JJ\mathbf{\Delta} JJ R_{\k} J= \k\,\mathbf{\Delta} \,\k.
\]
Hence $\mathbf{\Delta}_h $ and $\k\,\mathbf{\Delta}\, \k$ are antiunitarily equivalent so that they share the same spectral properties. Applying Corollary \ref{cor:defect} part 2)  yields that $\zeta_{1+ \mathbf{\Delta}_h}(0)$ is independent of $h$ so that $\zeta_{1+ \mathbf{\Delta}}(0)$ defines a conformal invariant.

\renewcommand\indexname{Notation Index}

\begin{theindex}

  \item $A_\th$, 5
  \item $\A_\th$, 6
  \item $\circ_\theta$, 10
  \item $\Cl_\theta^m(\T^n)$, 18
  \item $\Cl_\th^{\Z}(\T^n)$, 18
  \item $\Cl_\th^{\notin \Z}(\T^n)$, 18
  \item $CS^{m}_{\B}(\Z^n)$, 15
  \item $CS_\B^m(\R^n)$, 15
  \item $CS^{<-n}_{\B}(\Z^n)$, 15
  \item $CS^{\Z}_{\B}(\Z^n)$, 15
  \item $CS_\B^{\Z}(\R^n)$, 15
  \item $CS_{\B}(\Z^n)$, 15
  \item $CS_\B(\R^n)$, 15
  \item $CS_{\B}^{\notin \Z}(\Z^n)$, 15
  \item $CS_\B^{\notin \Z}(\R^n)$, 15
  \item $\cutoffint^{\cutoff}_{\R^n}$, 19
  \item $\cutoffint_{\R^n}$, 20
  \item ${\textstyle\cutoffsum_{\Z^n}^{\mathrm{cut-off}}}$, 20
  \item $\cutoffsum_{\Z^n}$, 20
  \item $\cutoffsumtor$, 26
  \item ${\textstyle\cutoffsum_\th }$, 28
  \item $\delta^\a$, 6
  \item $\delta_j$, 6
  \item $\bar \delta_j$, 11
  \item $ \delta {\TR}_\theta^P(A,B)$, 37
  \item $\Delta_{j}$, 7
  \item $\Delta^\a$, 7
  \item $e_\C$, 13
  \item $HS^m_{\B}(\R^n)$, 14
  \item $\iota_\th$, 9
  \item $L^1(\Z^n,c)$, 5
  \item $\Op_\th$, 9
  \item $p_\a$, 6
  \item $p_{\a,\b}^{(m)}$, 8
  \item $\Psi_\th^m(\T^n)$, 9
  \item $\Psi_\th(\T^n)$, 9
  \item $\Psi^{-\infty}_\th(\T^n)$, 9
  \item $q_{N}$, 7
  \item $QS^m_{\B}(\R^n)$, 15
  \item $\res$, 19
  \item $\res_\th$, 28
  \item $\restor$, 25
  \item $\Res_\th$, 28
  \item $\{\sigma, \tau\}_\theta$, 11
  \item $S^m_{\B}(\Z^n)$, 7
  \item $S^m_{\B}(\R^n)$, 7
  \item $S_{\B}(\Z^n)$, 7
  \item $S_{\B}(\R^n)$, 7
  \item $S^{-\infty}_{\B}(\Z^n)$, 7
  \item $S^{-\infty}_{\B}(\R^n)$, 7
  \item $\sg_{[m-j]}$, 15
  \item $[\sg]_{m-j}$, 15
  \item $\t$, 6
  \item $\bar \t$, 9
  \item $T_l$, 7
  \item $\TR_\th$, 28
  \item ${\TR}_\theta^P$, 37
  \item $U_k$, 5
  \item $\zeta_P$, 37
  \item $\zeta_\theta(A,P)(z)$, 37

\end{theindex}

\end{document}